\documentclass[a4alexpaper]{article}

\usepackage{amsmath,amsthm,latexsym,amssymb,color} %eufrak
\usepackage{bbm}
\usepackage{xargs}
\usepackage{enumerate}
\usepackage[textsize=footnotesize]{todonotes} % pour voir les commentaires
\usepackage{authblk}

\allowdisplaybreaks

\usepackage{xcolor}
\newtheorem{theorem}{Theorem}[section]

\newtheorem{lemma}[theorem]{Lemma}

\newtheorem{remark}[theorem]{Remark}
\newtheorem{definition}[theorem]{Definition}
\newtheorem{example}[theorem]{Example}
\newtheorem{assumption}[theorem]{Assumption}
\def\E{\mathbb{E}}
\def\P{\mathbb{P}}

\def\Rl{R_{\lambda}}
\def\Rd{\mathbb{R}^{d}}

\def\1{\mathds{1}}
\def\eqdef{:=}

\def\rme{\mathrm{e}}
\def\rmi{\mathrm{i}}
\def\rmd{\mathrm{d}}
\def\nset{\mathbb{N}}
\def\rset{\mathbb{R}}
\def\zset{\mathbb{Z}}
\newcommand{\indi}[1]{\mathbbm{1}_{{#1}}}
\newcommand{\indiacc}[1]{\mathbbm{1}_{\{#1\}}}
\newcommand{\ps}[2]{\left\langle #1, #2 \right\rangle}

\newcommandx{\as}[1][1=P]{\ensuremath{#1\, -\mathrm{a.s.}}}

\def\param{\theta}
\def\paramcur{\vartheta}
\newcommandx{\set}[2]{\{ {#1} \, ; \, {#2}\}}
\newcommand{\ball}[2]{\operatorname{B}(#1,#2)}
\newcommandx{\CPE}[3][1=]{{\mathbb E}^{#1}\left[\left. #2 \, \right| #3 \right]} %%%% esperance conditionnelle

\def\wrt{w.r.t.}
\def\iid{i.d.d.}

\begin{document}

\title{On stochastic gradient Langevin dynamics with dependent data streams
in the logconcave case
\thanks{All the authors were supported by The Alan Turing Institute, London under the EPSRC grant EP/N510129/1. N. H. C. and M. R. also enjoyed the support of the NKFIH (National Research, Development and Innovation Office, Hungary) grant KH 126505 and the ``Lend\"ulet'' grant LP 2015-6 of the Hungarian Academy of Sciences. Y. Z. was supported by The Maxwell Institute Graduate School in Analysis and its Applications, a Centre for Doctoral Training funded by the UK Engineering and Physical Sciences Research Council (grant EP/L016508/01), the Scottish Funding Council, Heriot-Watt University and the University of Edinburgh. }  }

\author[1]{M. Barkhagen}
\author[2]{N. H. Chau }
\author[3]{\'E. Moulines}
\author[2]{M. R\'asonyi }
\author[1, 4]{S. Sabanis}
\author[1]{Y. Zhang}

\affil[1]{\footnotesize School of Mathematics, The University of Edinburgh, UK.}
\affil[2]{\footnotesize Alfr\'ed R\'enyi Institute of Mathematics, Hungarian Academy of Sciences, Hungary.}
\affil[3]{\footnotesize Centre de Math\'ematiques Appliqu\'ees, UMR 7641, Ecole Polytechnique, France.}
\affil[4]{\footnotesize The Alan Turing Institute, UK.}

\renewcommand\Authands{ and }

\date{\today}

\maketitle

\begin{abstract}
We study the problem of sampling from a probability distribution $\pi$ on $\rset^d$ which has a
density \wrt\ the Lebesgue measure known up to a normalization factor $x \mapsto \rme^{-U(x)} / \int_{\rset^d} \rme^{-U(y)} \rmd y$.
We analyze  a sampling method based on the Euler discretization of the Langevin stochastic differential equations under
the assumptions that the potential $U$ is continuously differentiable, $\nabla U$ is Lipschitz, and $U$ is strongly concave.
We focus on the case where the gradient of the log-density cannot be directly computed but
unbiased estimates of the gradient from possibly dependent observations are available. This setting
can be seen as a combination of a stochastic approximation (here stochastic gradient) type algorithms with
discretized Langevin dynamics. We obtain an upper bound of the Wasserstein-2 distance between the law of the iterates of this algorithm and the target distribution $\pi$ with constants depending explicitly on the Lipschitz and strong convexity constants of the potential and the dimension of the space. Finally, under weaker assumptions on $U$ and its gradient but in the presence of independent observations, we obtain analogous results in Wasserstein-2 distance.
\end{abstract}

\section{Introduction}
\label{sec:introduction}
Sampling target distributions is an important topic in statistics and applied probability. %, for example, the computation of Bayesian estimators often requires sampling techniques.
In this paper, we are concerned with sampling from a distribution $\pi$ defined by
\[
\pi(A):=\int_A \rme^{-U(\theta)}\, \rmd \theta/\int_{\mathbb{R}^{d}} \rme^{-U(\theta)}\, \rmd \theta,\ \quad
A\in\mathcal{B}(\mathbb{R}^{d}),
\]
where $\mathcal{B}(\mathbb{R}^{d})$ denotes the Borel sets of $\mathbb{R}^{d}$
and $U:\mathbb{R}^{d}\to\mathbb{R}_+$ is continuously
differentiable.

One of the sampling schemes considered in this paper is the \emph{unadjusted} Langevin algorithm (a.k.a. Langevin Monte Carlo). The idea is to construct a Markov chain which is the Euler discretization of a continuous-time diffusion process that has an invariant distribution $\pi$.

We work on a fixed probability space $(\Omega,\mathcal{F},P)$ throughout the paper.
We consider the so-called overdamped Langevin stochastic differential equation (SDE)
\begin{equation} \label{eq1}
\rmd \theta_t = -h(\theta_t)dt+\sqrt{2}\rmd B_t,\
\end{equation}
with a (possibly random) initial condition $\theta_0$, where $h: = \nabla U$ and $(B_t)_{t \ge 0}$ is a $d$-dimensional Brownian motion. It is well-known that, under appropriate conditions, the Markov semigroup associated with the Langevin diffusion (\ref{eq1}) is reversible with respect to $\pi$, and the rate of convergence to $\pi$ is geometric in the total variation norm (see \cite{Mey1993b}, \cite[Theorem~1.2]{Rob1996}, and  \cite[Theorem~1.6]{asz}).
The Euler-Maruyama discretization scheme for SDE (\ref{eq1}), which is referred to as the unadjusted Langevin algorithm (ULA), is given by
\begin{equation}\label{aver}
\overline{\theta}^{\lambda}_0:=\theta_0,\quad
\overline{\theta}^{\lambda}_{n+1}:=\overline{\theta}^{\lambda}_n-\lambda
h(\overline{\theta}^{\lambda}_n)+\sqrt{2\lambda}\xi_{n+1},
\end{equation}
where $(\xi_n)_{n\in\nset}$ is a sequence of independent, standard $d$-dimensional Gaussian random variables, $\lambda > 0$ is the step size and $\theta_0$ is an $\mathbb{R}^{d}$-valued random variable denoting the initial values of both \eqref{aver} and \eqref{eq1}. Under appropriate assumptions on the step size $\lambda$ and the potential $U$, the
homogeneous Markov chain $(\overline{\theta}^{\lambda}_n)_{n \in \nset}$ converges to a distribution $\pi_{\lambda}$ which differs from $\pi$ but, for small $\lambda$, it is close to $\pi$ in an appropriate sense; see \cite{dalalyan}, \cite{dalalyan:karagulyan:2018},\cite{durmus-moulines}, and Section~\ref{sec_depdata}.

We now adopt a framework where the exact gradient $h$ is unknown, however one can observe at each iteration an unbiased estimator. Let $H:\mathbb{R}^{d}\times\mathbb{R}^{{m}}\to\mathbb{R}^{d}$ be a measurable
function and $X:=(X_n)_{n\in\nset}$ an $\mathbb{R}^{{m}}$-valued process
adapted to some given filtration $\mathcal{G}_{n}$, $n\in\mathbb{N}$ satisfying
\begin{equation}\label{timberlake}
h(\theta)=\E[H(\theta,X_n)],\quad \theta\in\mathbb{R}^d,\ n\geq 1,
\end{equation}
where the existence of the expectation is implicitly assumed.
Note that \eqref{timberlake} holds, in particular, when $(X_n)_{n\geq 1}$
is a strictly stationary process. Denoting by $\mu$ the (common) distribution of $X_n$, $n\geq 1$
we may write
\begin{equation}\label{timberlake1}
h(\theta)= \int H(\theta,x) \mu(\rmd x),
\end{equation}
in this case.
We also assume henceforth that $\theta_0$,
$\mathcal{G}_{\infty}$, $(\xi_n)_{n\in\nset}$ are independent.

For each $\lambda >0$, define an $\rset^{d}$-valued
random process $(\theta^{\lambda}_n)_{n\in\nset}$ by recursion:
\begin{equation}\label{nab}
\theta^{\lambda}_0:=\theta_0,\quad \theta^{\lambda}_{n+1}:=\theta^{\lambda}_n-\lambda H(\theta^{\lambda}_n,X_{n+1})+\sqrt{2\lambda}\xi_{n+1}.
\end{equation}
Such a sampling scheme is often called a stochastic gradient Langevin dynamics (SGLD) algorithm; see
\cite{WellingTeh}, \cite{dalalyan:karagulyan:2018} and \cite{teh}.
Data sequences $(X_n)_{n \in \nset}$ are in general not \iid, not even Markovian.
They may exhibit strong non-Markovian features as it is observed in various stochastic phenomena.
Stochastic approximation for dependent data sequences (gradient and Kiefer-Wolfowitz procedures)
has been successfully used in financial
applications, see \cite{laruelle,zhuang} and the references therein.
With these examples in mind, in the present paper we seek theoretical guarantees for the convergence of the closely related
SGLD procedure to ensure its validity for non-independent data sets, too.

The only instance we know of that provides results in such a setting is Theorem 4 of \cite{dalalyan:karagulyan:2018}.
The main condition of that result (Condition N in \cite{dalalyan:karagulyan:2018}) requires estimates on
the conditional bias and variance of the updating function with respect to the previous iterate of
the recursion \eqref{nab}, see Subsection \ref{subsec:assumptions-discussion} for extensive discussions.
In concrete examples it seems very difficult to determine the order of these quantities.
We follow a different path. Intuitively, if the signal $X_{n}$ is ``sufficiently ergodic'' then one should be able to
estimate the sampling error, without checking conditions on the conditional bias/variance of
specific objects.
We will assume a
certain mixing condition, \emph{conditional L-mixing} for the data sequence $(X_n)_{n\in \nset}$; see Section~\ref{sec:conditional-L-mixing}
below for technical details. Theorem \ref{main} is obtained which
guarantees an (essentially) optimal estimate in terms of the stepsize. Our approach involves
several new ideas which serve as a basis for further developments in the case of non-convex $U$, see \cite{nonconvex}.

The goal of this work is to establish an upper bound on the Wasserstein distance between the target distribution $\pi$ and its approximations $(\mathrm{Law}(\theta^{\lambda}_n))_{n\in\nset}$ generated by the SGLD algorithm \eqref{nab}. % Note that we allow the gradient to be estimated from a stationary but not necessarily \iid or even Markovian observation sequence.
This goal is achieved while the rate of convergence is improved with respect to the findings in
\cite{raginsky:rakhlin:telgarsky:2017}, see also \cite{xu}, \cite{dalalyanold} and \cite{dalalyan:karagulyan:2018}.
We stress that we prove the validity
of sampling procedures driven by SGLD \eqref{nab} within a framework where $(X_n)_{n \in \nset}$ are not assumed \iid
\ and hence $\theta^{\lambda}_n$ fails to be Markovian and related techniques are not applicable. Algorithms for
variance reduction of SGLD have been suggested by \cite{chatterji,xu}, however, we do not see for the moment
how these could be treated by our
methods here.

The paper is organized as follows. Section \ref{sec:conditional-L-mixing} recalls the theoretical concept of conditional $L$-mixing
which is required below for the process $(X_n)_{n \in \nset}$. This notion accommodates a large class of (possibly non-Markovian) processes.
In Section~\ref{sec:assumptions}, assumptions and main results are presented in the case where the process
$(X_n)_{n \in \nset}$ is conditionally $L$-mixing (Section~\ref{subsec:assumptions-L-mixing}) and
\iid\ (Section~\ref{subsec:assumptions-iid}), respectively. In Section~\ref{subsec:assumptions-discussion}, we discuss
the contributions of our work with respect to existing results reported in the literature. In Section \ref{sec_depdata}
and Subsection \ref{sec_nabdep}, the properties of (\ref{eq1}), (\ref{aver}), and (\ref{nab}) are analyzed. The proofs of
the main theorems are provided in Sections \ref{sec_depdataproof} and \ref{sec_indepdata}, while certain auxiliary results
are presented in Sections \ref{sec_app} and \ref{sec_iniga}.

\paragraph{Notations and conventions.} Scalar product in $\mathbb{R}^{d}$
is denoted by $\langle \cdot,\cdot\rangle$. We use $\| \cdot \|$ to denote
the Euclidean norm (where the dimension of the space may vary).
$\mathcal{B}(\mathbb{R}^{d})$ denotes the Borel $\sigma$- field of $\mathbb{R}^{d}$.
For each $x_0 \in \rset^d$ and $R\geq 0$, we denote $\ball{x_0}{R}:=\{x\in\mathbb{R}^{d}:\, \|x-x_0\|\leq R\}$, the closed
ball of radius $R$ centered at $x_0$.
For two sigma algebras $\mathcal{F}_1, \mathcal{F}_2$, we define $\mathcal{F}_1 \vee \mathcal{F}_2:= \sigma\left( \mathcal{F}_1 \cup \mathcal{F}_2\right).$
The expectation of
a random variable $X$ is denoted by $\E[X]$.
For any $m\geq 1$, for any $\mathbb{R}^{{m}}$-valued random variable $X$ and for any $1\leq p<\infty$, we set
$\Vert X\Vert_p:=\E^{1/p}[\|X\|^p]$. We denote by $L^p$ the set of $X$ with $\Vert X\Vert_p<\infty$.
The indicator function of a set $A$ is denoted by $\indi{A}$.  The Wasserstein distance of order $p \geq 1$ between two probability measures $\mu$ and $\nu$ on $\mathcal{B}(\mathbb{R}^{d})$ is defined by
\begin{equation}\label{w_dist}
W_p(\mu,\nu) = \left( \inf_{\pi \in \Pi(\mu,\nu)} \int_{\mathcal{X}} \Vert x-y\Vert^p
\rmd \pi(x,y)  \right)^{1/p},
\end{equation}
where $\Pi(\mu,\nu)$ is the set of couplings of $(\mu, \nu)$, see e.g. \cite{villani}.

\section{Conditional $L$-mixing}
\label{sec:conditional-L-mixing}

$L$-mixing processes and random fields
were introduced in \cite{laci1}. They proved to be useful in
tackling difficult problems of system identification, see e.g. \cite{laci6,laci4,laci5,laci7,q}. In
\cite{chau:kumar:2018}, in the context of stochastic gradient methods, the related
concept of \emph{conditional} $L$-mixing
was introduced. We now recall its definition below.

We consider the probability space $(\Omega,\mathcal{F},P)$, equipped
with a discrete-time filtration $(\mathcal{F}_n)_{n\in\nset}$ as well as with a decreasing sequence of sigma-fields $(\mathcal{F}_n^+)_{n\in\nset}$ such that $\mathcal{F}_n$ is
independent of $\mathcal{F}_n^+$, for all $n\in\nset$.

For a family $(Z_i)_{i\in I}$ of real-valued random variables (where the index set $I$ may have arbitrary cardinality),
there exists one and (up to a.s.\
equality) only one random variable $g = \mathrm{ess}\sup_{i\in I} Z_i$ such that:
\begin{enumerate}
	\item[(i)] $g \ge Z_i$, a.s. for all $i \in I$,
	\item[(ii)] if $g'$ is a random variable, $g' \ge Z_i$, a.s. for all $i \in I$ then $g' \ge g$ \as,
\end{enumerate}
see e.g.  \cite[Proposition~VI.1.1]{neveu}. %The random variable $g$ is denoted by $\mathrm{ess}\sup_{i\in I} Z_i$.

Fix an integer $d\geq 1$ and let $D\subset \mathbb{R}^{d}$ be a set of parameters. A measurable function
$U:\nset\times D\times\Omega\to\mathbb{R}^{{k}}$ is called a random field. We drop dependence on $\omega\in\Omega$ in the notation henceforth and write
$(U_n(\theta))_{n\in\nset, \theta\in D}$. A random
process $(U_n)_{n\in\nset}$ corresponds to a random field where $D$
is a singleton. A random field is $L^r$-\emph{bounded} for some $r\geq 1$
if
\[
\sup_{n\in\nset}\sup_{\theta\in D} \|U_n(\theta)\|_r<\infty.
\]

Let $U_n(\theta)\in L^{1}$, $n\in\nset$, $\theta\in D$ and $U_{n+m}^{i}$ is the $i$-th coordinate of $U_{n+m}$.
Define, for each $n\in\nset$, $i=1,\ldots,k$, and $\tau \in \nset$
\begin{align}
\label{eq:definition-M}
	\tilde{M}^{n}_r(U,i) &:= \mathrm{ess}\sup_{\theta\in D}\sup_{m \in\nset}
	\CPE[1/r]{|U_{n+m}^{i}(\theta)|^r}{\mathcal{F}_n},\\
\label{eq:definition-gamma}
\tilde{\gamma}^{n}_r(\tau,U,i)&:= \mathrm{ess}\sup_{\theta\in D}\sup_{m\geq\tau}
	\CPE[1/r]{|U^{i}_{n+m}(\theta)-\CPE{U^{i}_{n+m}(\theta)}{\mathcal{F}_{n+m-\tau}^+\vee \mathcal{F}_n}|^r}{
	\mathcal{F}_n},
\end{align}
and set	
\begin{equation}
\label{eq:definition-Gamma}
\text{$\tilde{\Gamma}^{n}_r(U,i) := \sum_{\tau= 0}^{\infty}\tilde{\gamma}^{n}_r(\tau,U,i)$,
${M}^{n}_{r}(U) :=\sum_{i=1}^{k}\tilde{M}^{n}_{r}(U,i)$, and $\Gamma^{n}_{r}(U) :=\sum_{i=1}^{k}\tilde{\Gamma}^{n}_{r}(U,i)$.}
\end{equation}
When necessary, the notations $M^{n}_r(U,D)$,
$\gamma^{n}_r(\tau,U,D)$ and $\Gamma^{n}_r(U,D)$ are used to emphasize dependence of
these quantities on the domain $D$ which may vary.

\begin{definition}[Conditional $L$-mixing] Let $r,s\geq 1$. We say that the random field
$(U_n(\theta))_{n\in\nset, \theta\in D}$ is
\emph{uniformly {conditionally} $L$-mixing} (UCLM) of order $(r,s)$
with respect to $(\mathcal{F}_n,\mathcal{F}_n^+)_{n\in \nset}$ if
$(U_n(\theta))_{n\in\nset}$ is adapted to
$(\mathcal{F}_n)_{n\in\nset}$
for any $\theta\in D$;
it is $L^r$-bounded;
and the sequences  $(M^n_r(U))_{n\in \nset}$, $(\Gamma^n_r(U))_{n\in\nset}$
are $L^s$-bounded. When this holds for all $r,s\geq 1$ then we call the random field simply ``uniformly $L$-mixing''.
In the case of stochastic processes (when $D$ is a singleton)
the terminology ``conditionally $L$-mixing process (of order $(r,s)$)'' is used.
\end{definition}

\begin{remark} {\rm The definition of conditional $L$-mixing in \cite{chau:kumar:2018} is slightly different from
the definition above but they are clearly equivalent.

Although we do not use the concept of $L$-mixing in the present paper it
is worth noting that the definition of a uniformly $L$-mixing process follows naturally from the above definition if one sets $d=1$, $n=0$ and
$\mathcal{F}_n$ is replaced by the trivial $\sigma$-algebra in the definitions of $M^n_r(U)$, $\gamma^{n}_r(\tau,U)$ and $\Gamma^n_r(U)$.
Then, one obtains deterministic $M_r(U)$, $\gamma_r(\tau,U)$, $\Gamma_r(U)$ and the required condition for these quantities becomes
$M_r(U)+\Gamma_r(U)< \infty$. For more details, one can consult \cite{chau:kumar:2018} and \cite{laci1}.}
\end{remark}

Let $(U_n)_{n \in \nset}$ be a conditionally $L$-mixing process. For later use, we also introduce the quantities for $r,s\geq 1$,
\begin{equation}\label{eq:process_constant_cond}
\mathcal{M}_{r}(U):=\sup_{n\in\nset}\E[\|U_n\|^r],\quad
\mathcal{C}_{r,s}(U):=\sup_{n\in\nset}\E[\{\Gamma^n_r(U)\}^s].
\end{equation}
The interpretation of $\mathcal{M}_r(U)$ is straightforward while $\mathcal{C}_{r,s}(U)$ serves as a certain
measure of dependence for the process $U$.

\begin{example}\label{iidcasse}
{\rm Let $(X_n)_{n \in \nset}$ be \iid\ random variables ($d=1$) and set $\mathcal{F}_n:=\sigma(X_k,\ k\leq n)$, $\mathcal{F}^+_n:=\sigma(X_k,\ k>n)$, $n\in\nset$.
If $\E[|X_0|^r]<\infty$ for any $r\geq 1$, then $(X_n)_{n \in \nset}$ is conditionally $L$-mixing with respect to $(\mathcal{F}_n,
\mathcal{F}^+_n)_{n\in\nset}$. Moreover,}
\begin{equation}\label{process_constant}
\mathcal{M}_{r}(X)=\E[|X_0|^r],\quad
\mathcal{C}_{r,s}(X)=\E^{s/r}[|X_0-\E[X_0]|^r]  \, r,s\geq 1.
\end{equation}
\end{example}

\begin{example}
{\rm Let us consider, for example, a functional of a linear process  $U:=\{ U_n(\theta) \}_{n \in \nset}$, such that
\begin{equation}
\label{eq:definition-X}
U_n(\theta):=G(\theta,X_n) \,, \quad X_n:=\sum_{k=0}^{\infty}a_k \varepsilon_{n-k},,
\end{equation}
with scalars $( a_k)_{k \in \nset}$, some sequence $(\varepsilon_k)_{k\in\zset}$ of \iid\ $\mathbb{R}$-valued
random variables satisfying $\|\varepsilon_0\|_p < \infty$ for all $p \ge 1$ and $G: \rset \times \rset \to \rset$ a function satisfying
\[
|G(\theta,x) - G(\theta',x')| \leq L_1 |\theta - \theta'| + L_2 |x-x'| \, .
\]
Let $\mathcal{G}_n = \sigma( \varepsilon_j, j\le n)$, and $\mathcal{G}^+_n = \sigma( \varepsilon_j, j > n)$ for $n \in \nset$.
If we further assume that $|a_k|\leq c (1+k)^{-\beta}$, $k\in\nset$
for some $c>0$, $\beta>3/2$
then the argument of \cite[Lemma 4.2 ]{chau:kumar:2018} shows that $(X_n)_{n\in \nset}$ is a conditionally $L$-mixing process with respect to $(\mathcal{G}_n,\mathcal{G}_n^+)_{n\in \nset}$.
Applying Lemma~\ref{below} below with $\vartheta=0$ shows that for all $j\in\nset$, $M_r^n(U,\ball{0}{j})) \leq L_1 j+ L_2 M_r^n(X) + |G(0,0)|$ and $\Gamma_r^n(U,\ball{0}{j}) \leq 2 L_2 \Gamma_r^n(X)$.}
\end{example}

%In order to help the reader further understand how extensive is the class of processes which fall under our framework, we provide the following remark.

\begin{remark}
{\rm If $(X_n)_{n\in \nset}$ is a conditionally $L$-mixing process
with respect to $(\mathcal{F}_n,\mathcal{F}_n^+)_{n\in \nset}$ then so is $(F(X_n))_{n \in \nset}$ for any Lipschitz-continuous
function $F$, see \cite[Remark~2.3]{chau:kumar:2018}. Finally, we know
from  \cite[Example~7.1]{balazs} that a broad class of functionals of geometrically ergodic Markov chains have the $L$-mixing property. It is possible
to show, along the same lines, the conditional $L$-mixing property of
these functionals, too.}
\end{remark}

\section{Assumptions and main results}
\label{sec:assumptions}

\subsection{Dependent data}
\label{subsec:assumptions-L-mixing}
%Let us recall the filtration $\mathcal{G}_n:=\sigma(\varepsilon_j,\, j\leq n,\, j\in\zset)$, $n \in \nset$
%and define also the decreasing sequence of sigma-algebras
%\begin{equation}
%\label{eq:definition-G-n}
%\mathcal{G}^+_n:=\sigma(\varepsilon_j,\, j\geq n+1,\, j\in\zset),\ n\in\nset.
%\end{equation}
%where $(\varepsilon_n)_{n\in\zset}$ is the noise sequence generating
%$(X_n)_{n\in\zset}$, see \eqref{mocsing} above.

\begin{assumption}\label{assump:lll}
Let $\mathcal{G}_{0}:=\{\emptyset,\Omega\}$.
The process $(X_n)_{n\in\nset}$ is
conditionally $L$-mixing with respect to $(\mathcal{G}_n,\mathcal{G}_n^+)_{n\in\nset}$,
where $(\mathcal{G}^{+}_{n})_{n \in \nset}$ is some decreasing sequence of sigma-fields with $\mathcal{G}_{n}$
independent of $\mathcal{G}^{+}_{n}$ for all $n\in\mathbb{N}$. Furthermore, let $\|\theta_0\|_p < \infty$ for all $p\geq 1$.
\end{assumption}
For $(x,\theta) \in \rset^m \times \rset^d$, we denote $H(x,\theta)= [H^1(x,\theta),\dots,H^d(x,\theta)]^T$.
\begin{assumption}\label{assump:lip}  There exist constants $L^i_1, L^i_2>0$, $i \in \{1,\dots, d\}$ such that for all
$\theta,\theta' \in \mathbb{R}^{d}$ and $x, x' \in \mathbb{R}^m$, $ |H^i(\theta,x)-H^i(\theta',x') |\leq L^i_1 \|\theta-\theta'\|+L^i_2\|x - x'\|$.
\end{assumption}
We set
\begin{equation}
\label{eq:definition-L-1&2}
L_1= \sum_{i=1}^d L_1^i \quad  \text{and} \quad  L_2= \sum_{i=1}^d L_2^i\,.
\end{equation}
Note that, under Assumption~\ref{assump:lip}, for any $(x,\theta) \in \rset^m \times \rset^d$ we get
\[
\|H(x,\theta) - H(x,\theta') \| \leq L_1 \|\theta - \theta'\| + L_2 \| x - x'\| \,.
\]
Assumption \ref{assump:lll} implies, in particular, that $\|X_0\|\in L^r$, for any $r \geq 1$, thus, under Assumption \ref{assump:lll} and \ref{assump:lip}, $h(\theta):=\E[H(\theta,X_0)]$, $\theta\in\mathbb{R}^{d}$, is indeed well-defined.
\begin{assumption}\label{diss} There is a constant
$a > 0$ such that for all $\theta,\theta'\in\mathbb{R}^{d}$ and $x\in\mathbb{R}^{{m}}$,
\begin{equation}\label{montre}
\ps{\theta-\theta'}{H(\theta,x)-H(\theta',x)} \geq a\|\theta-\theta'\|^2.
\end{equation}
\end{assumption}
%\begin{remark}	{\rm For a differentiable function $f:\mathbb{R} \to \mathbb{R}$, it is known that $f$ is convex if and only if $f'$ is monotone nondecreasing. Similarly, a real-valued function $f$ on a Hilbert space $H$ is convex if and only if 	\[\left\langle  y -x, \nabla f(y) - \nabla f(x) \right\rangle \ge 0, \qquad \text{ for all } x, y \in H.\] 	The monotonicity condition (\ref{montre}) clearly holds if, for all $x \in \mathbb{R}^{\mathpzc{m}}$, $\theta\to H(\theta,x)$ is the derivative of a function which is strongly convex in $\theta$, \emph{uniformly} in $x\in\mathbb{R}^{\mathpzc{m}}$, see pages 63--66 in \cite{nesterov}.} \end{remark}

Two important properties immediately follow from Assumptions \ref{assump:lip} and \ref{diss}.
\begin{enumerate}
\item[\bf(B1)]
%The function $h$ is $L$-smooth: there exists a non-negative constant $L$ such that,
For all $\theta,\theta' \in \mathbb{R}^{d}$, $\| h(\theta)-h(\theta')\| \leq L_1\|\theta-\theta' \|$
\item[\bfseries(B2)]
There exists a constant $a>0$ such that, for all $\theta,\theta' \in \mathbb{R}^d$,$ \ps{\theta-\theta'}{h(\theta)-h(\theta')} \geq a \| \theta -\theta' \|^2$.
\end{enumerate}
\cite[Theorem~2.1.12]{nesterov} shows that, under these assumptions, for all $\theta,\theta' \in \mathbb{R}^{d}$,
\begin{equation}\label{B2-alt}
\ps{\theta-\theta'}{h(\theta)-h(\theta')} \geq \tilde{a} \| \theta -\theta' \|^2 +\frac{1}{a+L_1}\|h(\theta)-h(\theta')\|^2,
\end{equation}
where we have set
\begin{equation}
\label{eq:definition-tilde-a}
\tilde{a} = \frac{aL_1}{a+L_1} \,.
\end{equation}
%{\bf (*)} The function $h$ is $L$-smooth: there exists a non-negative constant $L$ such that \[ \| h(\theta_1)-h(\theta_2)\| \leq L\|\theta_1-\theta_2 \|, \quad \forall \theta_1,\theta_2 \in \mathbb{R}^{d}. \]
%
%{\bf (**)} The function $h$ is strongly convex: there exists a constant $a>0$ such that \[ \langle \theta_1-\theta_2,h(\theta_1)-h(\theta_2)\rangle \geq a \left(\| \theta_1 -\theta_2 \|^2+ \| h(\theta_1) -h(\theta_2) \|^2\right), \quad \forall \theta_1,\theta_2 \in \mathbb{R}^{d}. \]

Our aim initially is to estimate
$\Vert\theta^{\lambda}_n-\overline{\theta}^{\lambda}_n\Vert_2$,
uniformly in $n \in \nset$. To begin with, an example is presented where explicit calculations are possible.

\begin{example} \label{example_rate_one_half} {\rm Let $d:=1$, $H(\theta,x):=\theta+x$,
$(X_n)_{n\in\zset}$
		be a sequence of satisfying \eqref{eq:definition-X} with $(\epsilon_j)_{j \in \zset}$ an independent sequence of standard Gaussian random variables
		independent of $(\xi_n)_{n\in\nset}$; and $|a_k| \leq c (1+k)^{-\beta}$, $k \in \nset$ for some $\beta > 3/2$ and
\begin{equation}
\label{eq:lower-bound-spectral-density}
0<m \eqdef \inf_{\mu \in [-\pi,\pi]} \left| \sum_{k=0}^\infty a_k \rme^{-\rmi \mu k}\right| \leq
\sup_{\mu \in [-\pi,\pi]} \left| \sum_{k=0}^\infty a_k \rme^{-\rmi \mu k}\right| \leq M<\infty \,.
\end{equation}
We observe that the function $H$ satisfies Assumptions \ref{assump:lip} and \ref{diss}. Take $\theta_0:=0$.
		It is straightforward to check that, for any $\lambda \in (0,1)$,
		\[
		\overline{\theta}^{\lambda}_n-\theta^{\lambda}_n= \sum_{j=0}^{n-1}
		(1-\lambda)^j \lambda X_{n-j},
		\]
		which clearly has variance
		\[
		\E[(\overline{\theta}^{\lambda}_n-\theta^{\lambda}_n)^2]= \frac{\lambda^2}{2\pi} \int_{-\pi}^{\pi} \left| \sum_{k=0}^\infty a_k \rme^{-\rmi k \mu} \right|^2 \left| \sum_{k=0}^{n-1} (1-\lambda)^k \rme^{-\rmi k \mu} \right|^2 \rmd \mu
		\]
		It follows that, using \eqref{eq:lower-bound-spectral-density} and the Parseval-Plancherel Theorem
		\[
		m \sqrt{\frac{\lambda \{1-(1-\lambda)^{2n}\}}{2-\lambda}} \leq \|\overline{\theta}^{\lambda}_n-\theta^{\lambda}_n\|_2 \leq M \sqrt{\frac{\lambda  \{1-(1-\lambda)^{2n}\}}{2-\lambda}}.
		\]
		This shows that the best estimate we may hope to obtain for $\sup_{n \in \nset} \|\overline{\theta}^{\lambda}_n-\theta^{\lambda}_n\|_2$ is of the order $\sqrt{\lambda}$.
Theorem \ref{main} below achieves this bound asymptotically as $p\to\infty$.}
\end{example}

Our main results may be stated as follows.
\begin{theorem}\label{main}
	Let Assumptions \ref{assump:lll}, \ref{assump:lip} and \ref{diss} hold. For every even number $p\geq 4$ and $\lambda < \bar{\lambda},$ where
\begin{equation}
\label{eq:definition-bar-lambda}
\bar{\lambda}:=\frac{2}{a+L_1},
\end{equation}
there exists $C_0(p)>0$ such that
\begin{equation} \label{eq:thrmrate1}
\|\theta^{\lambda}_n-\overline{\theta}^{\lambda}_n\|_2
\leq C_0(p)\lambda^{\frac{1}{2} - \frac{1}{p}}, \qquad n\in\nset
\end{equation}
holds for a constant $C_0(p)$ that is explicitly given in the proof. It depends only on $a$, $L_1$, $L_2$, $d$, $p$ and on
the process $(X_n)_{n \in \nset}$ through the quantities defined in \eqref{eq:process_constant_cond}.
\end{theorem}
\begin{proof}
The proof of this theorem is postponed to Section~\ref{sec_proof_main}.
\end{proof}

%\erici{the syntax $C_0(p)$ is a bit strange, it depends on $p$ and all the other constants of the problem... I agree that $L_1$, $L_2$, $a$ are constants of the problem, whereas $p$ can be changed; It would be great ifwe can say something, even imprecise, on the way this constant depends on $p$. Is it polynomial, exponential ?}

The next result relates our findings in Theorems \ref{main} to the problem of sampling from
the probability law $\pi$.
\begin{theorem}\label{dros}
Let Assumptions \ref{assump:lll}, \ref{assump:lip} and \ref{diss} hold and let $\bar{\lambda}$ be given by \eqref{eq:definition-bar-lambda}.  For each $\kappa>0$, there exist constants $c_1(\kappa),c_2(\kappa)>0$ such that, for each
	$0<\epsilon\leq \mathrm{e}^{-1}$ one has
	\[
	W_2(\mathrm{Law}(\theta^{\lambda}_n),\pi)\leq \epsilon
	\]
	whenever $\lambda <\bar{\lambda}$ satisfies
	\begin{equation}\label{sharp}
	\lambda=  c_1(\kappa)\epsilon^{2+\kappa}\mbox{ and }n\geq
	\frac{c_2(\kappa)}{\epsilon^{2+\kappa}}\ln(1/\epsilon),
	\end{equation}
where $c_1(\kappa), c_2(\kappa)$ (given explicitly in the proof) depend only on $\kappa$,
$d$, $a$, $L_1$, $L_2$ and on the process $(X_n)_{n \in \nset}$ through the quantities defined in \eqref{eq:process_constant_cond}.
\end{theorem}
\begin{proof}The proof of this theorem is postponed to Section \ref{sec_depdatacor}.
\end{proof}

\subsection{Independent data}
\label{subsec:assumptions-iid}
When the data sequences $(X_n)_{n\in\zset}$ are \iid, then the full rate is recovered under more relaxed conditions for the unbiased estimator of the gradient of $U$. More concretely, one may assume the following:

\begin{assumption}\label{iidlocLip}
There exist positive constants $L_1$, $L_2$ and $\rho$ such that, for all $x,x'\in\mathbb{R}^m$ and $\theta, \theta'\in\mathbb{R}^d$,
\begin{align*}
  \|H(\theta,x)- H(\theta',x)\| & \le L_1(1+\|x\|)^{\rho}\|\theta-\theta'\|, \\
  \|H(\theta,x)- H(\theta,x')\| & \le L_2(1+\|x\|+\|x'\|)^{\rho}(1+ \|\theta\|)\|x-x'\|.
\end{align*}
\end{assumption}

\begin{assumption}\label{iid}
The process $(X_n)_{n \in \nset}$ is \iid with $\|X_0\|_{2(\rho +1)}$ and $\|\theta_0\|_2$ being finite.
\end{assumption}

\begin{assumption}\label{iidcnv}
There exists a mapping $A: \mathbb{R}^m \rightarrow \mathbb{R}^{d\times d}$ such that
\[
 \langle y, A(x)y\rangle \ge 0, \mbox{ for any } x,y\in \mathbb{R}^d  \qquad \qquad\mbox{ (positive semidefinite)}
\]
and, for all $\theta, \theta' \in \mathbb{R}^d$ and $x \in \mathbb{R}^m$,
 \[
 \langle \theta - \theta', H(\theta,x) - H(\theta', x)\rangle \geq \langle \theta - \theta', A(x)(\theta - \theta') \rangle
 \]
with the smallest eigenvalue of the matrix $\E[A(X_0)]$ being a positive real number which is denoted by $a$.
\end{assumption}

It is clear then that properties \textbf{(B1)} and \textbf{(B2)} are still valid for the gradient $h$ of $U$, with the only difference that the Lipschitz constant in \textbf{(B1)} is given by $L_1\E[(1+\|X_0\|)^{\rho}]$. This allows us to obtain the following result.
\begin{theorem} \label{iid-cor}
Let Assumptions \ref{iidlocLip}, \ref{iid} and \ref{iidcnv} hold and let $\bar{\lambda}$ be given by \eqref{eq:definition-bar-lambda}. There exist constants $c_1, c_2>0$ such that, for each $0<\epsilon \leq 1/2$,
\[
W_2(\mathrm{Law}(\theta^{\lambda}_n),\pi)\leq \epsilon.
\]
whenever $\lambda \le \min \Big(a/2L_1^2\E[(1+\|X_0\|)^{2\rho}],\, 1/a \Big)$ satisfies
	\begin{equation}\label{sharp_iid}
	\lambda\leq  c_1\epsilon^{2}\mbox{ and }n\geq
	\frac{c_2}{\epsilon^{2}}\ln(1/\epsilon),
	\end{equation}
where $c_1, c_2$ (given explicitly in the proof) depend only on
$d$, $a$, $\E[\|X_0\|^{2\rho+2}]$, $L_1$ and $L_2$. If $\rho=0$ in Assumption \ref{iidlocLip}, then the above results are true for $\lambda \le 1/2\min(L_1^{-1}, \bar{\lambda})$.
\end{theorem}
\begin{proof}
The proof of this Theorem is postponed to Section~\ref{sec_indepdata}.
\end{proof}

\subsection{Discussion}
\label{subsec:assumptions-discussion}
\paragraph{Rate of Convergence:}
Theorem \ref{dros} significantly improves on some of the results in \cite{raginsky:rakhlin:telgarsky:2017} in certain cases,
compare also to \cite{xu}. In \cite{raginsky:rakhlin:telgarsky:2017}
the monotonicity assumption \eqref{montre} is not imposed, only a dissipativity
condition is required and a more general recursive scheme is investigated. However, the input sequence $(X_n)_{n\in\nset}$ is assumed \iid\ In that setting,
		\cite[Theorem 2.1]{raginsky:rakhlin:telgarsky:2017} applies to \eqref{nab} (with the choice $\delta=0$,
		$\beta=1$ and $d$ fixed, see also the last paragraph of Subsection 1.1
		of \cite{raginsky:rakhlin:telgarsky:2017}), which implies that
		\[
		W_2(\mathrm{Law}(\theta^{\lambda}_n),\pi)\leq \epsilon
		\]
		holds whenever
		$\lambda\leq c_3(\epsilon/\ln(1/\epsilon))^4$ and
		$n\geq \frac{c_4}{\epsilon^4}\ln^5(1/\epsilon)$ with some $c_3,c_4>0$.
		For the case of \iid $(X_n)_{n \in \nset}$ see also the very
recent \cite{alex}. Our results provide the sharper estimates \eqref{sharp} in a setting where $(X_n)_{n \in \nset}$ may have
dependencies.
%Thus, a somewhat improved rate is obtained in comparison to 1/2-$\epsilon$ in \eqref{eq:thrmrate1} but for a
%significantly smaller class of dependent data.

\paragraph{Comparison with \cite{dalalyan:karagulyan:2018}}:
One notes, further, that a noisy Langevin Monte Carlo algorithm (nLMC) with inaccurate drift is proposed in
\cite{dalalyan:karagulyan:2018}, where the drift is assumed to be a linear combination of the original gradient and of
random noise represented by a dependent sequence of random vectors with non-zero means. Thus, a particular
form of dependency is included in this approach. A convergence result, \cite[Theorem 4]{dalalyan:karagulyan:2018}, in
Wasserstein-2 distance between nLMC and the target distribution $\pi$ is provided, which is in agreement with our findings,
i.e.\ rate of convergence equal to 1/2 is given when the bias term is eliminated.

%\cite[Theorem~4]{dalalyan:karagulyan:2018}  provides $W_2$-convergence for (possibly) dependent data-streams.
In \cite[Condition~N]{dalalyan:karagulyan:2018}, two quantities enter into play: the upper bound $L^2$-norm of the conditional bias,
$\E[ \| \CPE{H(\theta_k^\lambda,X_{k+1})}{\theta_k^\lambda} - h(\theta_k^\lambda) \|^2]$ and the variance
$\E[\|H(\theta_k^\lambda,X_{k+1})-\CPE{H(\theta_k^\lambda,X_{k+1})}{\theta_k^\lambda}\|^2]$.
We stress that, when the process $(X_k)_{k \in \nset}$ is actually dependent, $\theta_k^\lambda$ and $X_{k+1}$
are dependent and therefore $\CPE{H(\theta_k^\lambda,X_{k+1})}{\theta_k^\lambda} \ne h(\theta_k^\lambda)$ in general.
With the exception of a few very simple cases, a precise computation of conditional bias $\CPE{H(\theta_k^\lambda,X_{k+1})}{\theta_k^\lambda} - h(\theta_k^\lambda)$ (or of a tight upper bound for the $L^2$ norm of this quantity) is out of reach.
Using \eqref{timberlake} and Assumption~\ref{assump:lip}, we get that, for all $k \in \nset$,
\[
\| \CPE{H(\theta_k^\lambda,X_{k+1})}{\theta_k^\lambda} - h(\theta_k^\lambda) \|^2 \leq L_2^2 \int \CPE{\|X_k-x\|^2}{\theta_k^\lambda} \mu(\rmd x) \, ,
\]
where $\mu$ denotes the common law of the $X_{k}$. This
implies that $\E[\| \CPE{H(\theta_k^\lambda,X_{k+1})}{\theta_k^\lambda} - h(\theta_k^\lambda) \|^2] \leq \delta^2 d$ with
\[
\delta^2 \leq 2 d^{-1}  L_2^2 \left\{ \mathcal{M}_2(X) + \int \|x\|^2 \mu(\rmd x) \right\}
\]
Similarly, using again Assumption~\ref{assump:lip}, we get
\begin{multline*}
\E[\|H(\theta_k^\lambda,X_{k+1})-\CPE{H(\theta_k^\lambda,X_{k+1})}{\theta_k^\lambda}\|^2]
\leq 2 \E[\| H(\theta_k^\lambda,X_{k+1}) - H(\theta_k^\lambda,0) \|^2]
\\+ 2\E[\| \CPE{ H(\theta_k^\lambda,X_{k+1}) - H(\theta_k^\lambda,0)}{\theta_k^\lambda} \|^2] \leq 4 L_2^2 \mathcal{M}_2(X) =: \sigma^2 d\,.
\end{multline*}
Our assumptions therefore imply \cite[Condition~N]{dalalyan:karagulyan:2018} but the conclusions that we reach in Theorems~\ref{main} and \ref{dros} are sharper (note that the
bias term in \cite[Theorem~4]{dalalyan:karagulyan:2018} does not vanish as $\lambda \downarrow 0^+$).

\paragraph{Choice of step size:}  It is pointed out in \cite{Rob1996} that the ergodicity property of (\ref{aver}) is sensitive to the step size $\lambda$. Moreover, \cite[Lemma 6.3 ]{mattingly} gives an example in which the Euler-Maruyama discretization is transient. As pointed out in \cite{mattingly}, under discretization, the minorization condition is insensitive with appropriate sampling rate while the Lyapunov condition may be lost. An invariant measure exists if the two conditions hold simultaneously, see \cite[Theorem 7.3]{mattingly} and also \cite[Theorem 3.2]{Rob1996} for similar discussions. In this work, an approach similar to
\cite{dalalyan} is chosen, in that strong convexity of $U$ is assumed together with Lipschitzness of its gradient and, thus, the ergodicity of (\ref{aver}) is obtained.

\section{Proof of main results: dependent data}\label{sec_depdataproof}
\subsection{The Langevin SDE and its discretization: the strongly convex case}\label{sec_depdata}

Before proceeding to the demonstrations of the main results, we recall
here some recent results on the diffusion of Langevin and its discretization for strongly convex
potentials. All the results presented here are classic and can be found in either \cite{durmus-moulines} or \cite{dalalyan:karagulyan:2018}.

By   \cite[Theorem~2.1.8]{nesterov}, $U$ has a unique minimum at
some point $\theta^*\in\mathbb{R}^{d}$.
Note that due to the Lipschitz condition {\bf (B1)}, the SDE (\ref{eq1}) has a unique strong solution. It is a well-known result that the Langevin SDE \eqref{eq1} admits a unique invariant measure $\pi$. By \cite[Theorem~4.20]{kar1991}, one constructs the associated strongly Markovian semigroup $(P_t)_{t\geq0}$ given for all $t\geq 0$, $x \in \mathbb{R}^{d}$ and $\mathrm{A}\in \mathcal{B}(\mathbb{R}^{d})$ by $P_t(x,\mathrm{A})=P(\theta_t \in\mathrm{A}|\theta_0=x)$.

The following lemma from \cite{durmus-moulines} with adapted statement provides the explicit bound of the second moment of the Langevin diffusion, which allows the analysis of the Wasserstein-2 distance between $\pi$ and the aforementioned sampling algorithms.

\begin{lemma}[Proposition 1 in \cite{durmus-moulines}]\label{prop2}
	Let Assumptions \ref{assump:lip} and \ref{diss} hold and thus {\bf (B1)}, {\bf (B2)} are thereby implied.
	\begin{enumerate}
		\item[(i)] For all $t\geq 0$ and $\paramcur \in \Rd$,
		\begin{align*}
		\int_{\Rd}\|\paramcur-\theta^*\|^2P_t(\param,\rmd \paramcur)\leq \|\param-\theta^*\|^2 \rme^{-2at}+(d/a)(1-\rme^{-2at}).
		\end{align*}
		\item[(ii)] The stationary distribution $\pi$ satisfies
		\begin{align*}
		\int_{\Rd}\|\paramcur-\theta^*\|^2\pi(\rmd \paramcur) \leq d/a.
		\end{align*}
	\end{enumerate}
\end{lemma}

For a fixed step size $\lambda \in (0,1]$, consider the Markov kernel $R_{\lambda}$ given for all $A \in \mathcal{B}(\mathbb{R}^{d})$ and $\theta \in \mathbb{R}^{d}$ by
\begin{equation}\label{kernel_aver}
R_{\lambda}(\theta,A) = \int_{A}{ (4\pi \lambda)^{-d/2} \exp\left(-(4\lambda)^{-1}\left\|\paramcur-\theta + \lambda  h(\theta) \right\|^2  \right)\rmd \paramcur.}
\end{equation}
The discrete-time Langevin recursion (\ref{aver}) defines a time-homogeneous Markov chain, and for any $n\ge 1$, and for any bounded
(or non-negative) Borel function $f:\mathbb{R}^{d}\to\mathbb{R}$,
\[\CPE{f(\overline{\theta}^{\lambda}_n)}{\overline{\theta}^{\lambda}_{n-1}} = R_{\lambda}f(\overline{\theta}^{\lambda}_{n-1}) = \int_{\mathbb{R}^{d}}{f(\paramcur)R_{\lambda}(\overline{\theta}^{\lambda}_{n-1},\rmd \paramcur)}.\]

%We say that a function $V:\mathbb{R}^{d} \to [0, \infty)$ satisfies a Foster-Lyapunov drift condition for the family of Markov kernels $\set{R_{\lambda}}{\lambda \in \rset_+}$ if there exist constants $\lambda_0 >0, \alpha \in [0,1), c>0$ such that for all $\lambda \in (0,\lambda_0]$,
%\begin{equation}\label{foster}
%R_{\lambda}V \le \rho^{\lambda} V + \lambda c.
%\end{equation}

Lemma \ref{lem:properties-R-lambda} below is also a result from \cite{durmus-moulines} and along with Lemma \ref{prop2} are presented here for completeness by using the notation of this article. In particular, Lemma \ref{lem:properties-R-lambda} states that $R_{\lambda}$ admits a unique stationary distribution $\pi_{\lambda}$, which may differ from $\pi$. %(see Section 3.1 in \cite{Rob1996} and Theorem 15.0.1 in \cite{Mey1993b}). Nevertheless, one can also refer to Proposition 3 in \cite{durmus-moulines}, where the unique $\pi_{\lambda}$ is obtained by showing that $R_{\lambda}$ is a contraction mapping.

%The following lemmata show that (\ref{foster}) holds for suitable polynomial Lyapunov functions.

\begin{lemma}\label{lem:properties-R-lambda}
Let Assumption \ref{diss} hold and thus {\bf (B2)} is thereby implied. Then, for all $\lambda <\bar{\lambda}$, where $\bar{\lambda}$ is defined in \eqref{eq:definition-bar-lambda}, the following hold:
\begin{enumerate}[(i)]
	\item For all $\param \in \Rd$, $n \geq 1$,
	\begin{align*}
	\int_{\Rd}\|\paramcur-\theta^*\|^2R_{\lambda}^n(\param,\rmd \paramcur)\leq (1-2\tilde{a}\lambda)^{n}\|\param-\theta^*\|^2+(d/\tilde{a})(1-(1-2\tilde{a}\lambda)^n).
	\end{align*}
 	\item The Markov kernel $R_\lambda$ has a unique stationary distribution $\pi_{\lambda}$  which satisfies
	\begin{align*}
	\int_{\Rd}\|\param-\theta^*\|^2\pi_{\lambda}(\rmd \param) \leq d/\tilde{a}.
	\end{align*}
where $\tilde{a}$ is defined in \eqref{eq:definition-tilde-a}.
	\item For all $\param \in \Rd$, $n \geq 1$,
	\begin{align*}
	W_2(\delta_\param \Rl^n, \pi_{\lambda})\leq \rme^{-\tilde{a}\lambda n}\sqrt{2}( \|\param-\theta^*\|^2+d/\tilde{a})^{1/2}.
	\end{align*}
	\item \label{item:contraction} For all $n\in\nset$ and square-integrable $\mathbb{R}^d$-valued random variables $\eta_1,\eta_2$ with $\sigma(\eta_1,\eta_2)$ independent
of $\xi_k$, $k\in\nset$
	\begin{align*}
\E[\| \overline{\theta}^{\lambda}_n(1)-\overline{\theta}^{\lambda}_n(2)\|^2]\leq \rme^{-2\tilde{a}\lambda n}\E[\|\eta_1-\eta_2\|^2],
	\end{align*}
where $\overline{\theta}^{\lambda}_n(i)$, $i=1,2$ denote the solutions of the recursion \eqref{aver} with
the respective initial conditions $\theta_0=\eta_i$, $i=1,2$.
\end{enumerate}
\end{lemma}
\begin{proof} For the first three statements, see \cite[Propositions~2 and 3]{durmus-moulines}. For \ref{item:contraction}, see the proof of  \cite[Proposition~3]{durmus-moulines}.
%proves (iv) above in the case where $\mu_1=\delta_{\theta_1}$, $\mu_2=\delta_{\theta_2}$
%for some $\theta_1,\theta_2\in\mathbb{R}^d$. The general case follows from this as every law can be
%represented as a mixture of Dirac measures.
\end{proof}

Note that by Lemma \ref{lem:properties-R-lambda}, a Foster-Lyapunov type drift condition
is satisfied with $V_1(\theta): = \|\theta-\theta^*\|^2$, which yields that $\sup_{n \geq 0} \| \overline{\theta}^{\lambda}_n \|_2 < \infty$. This allows the analysis of the convergence between the recursive scheme \eqref{aver} and the stationary distribution $\pi$ in Wasserstein-$2$ distance (see Theorem \ref{thm6} below). However, in order to obtain the rate of convergence between \eqref{aver} and the SGLD scheme \eqref{nab}, the finiteness of higher moments is required. In the following Lemma, one obtains the drift condition with $V_p(\param): = \|\param-\theta^*\|^{2p}$, $p\in\nset\setminus\{0\}$.

\begin{lemma}\label{lem_foster} Let Assumptions \ref{assump:lll}, \ref{assump:lip} and \ref{diss} hold.
For any integer $p \geq 1$ , let $V_p(\param):=\|\param-\theta^*\|^{2p}$. Then, the process $\overline{\theta}^{\lambda}$ satisfies, for any $n \in \nset$ and $\lambda <\bar{\lambda}$, where $\bar{\lambda}$ is defined in \eqref{eq:definition-bar-lambda},
	\begin{equation}\label{ly_aver}
\CPE{V_p(\overline{\theta}^{\lambda}_{n+1})}{\overline{\theta}_{n}^{\lambda}} \leq \rho_{\lambda} V_p(\overline{\theta}^{\lambda}_{n}) +
\lambda C'(p) ,
	\end{equation}
where $\rho_{\lambda} = 1-\tilde{a}\lambda \in (0,1)$ and
\begin{align}
%\label{eq:definition-M}
%\overline{M} &= \sqrt{d(2p-1)p2^{2p-1}/\tilde{a}} \\
\label{eq:definition-C}
C'(p) &:=  d^p(2p-1)^p p^p2^{p(2p-1)}\tilde{a}^{1-p} + (2p-1)p 2^{3p-2} 2^{2p}d^p p^{\frac{3}{2}p}.
\end{align}
Moreover,
\begin{equation}\label{twinkle2}
\sup_{\lambda <\bar{\lambda}} \sup_{n} \E[V_p(\overline{\theta}^{\lambda}_n)] \leq \E[ V_p(\theta_0)] + C'(p)/\tilde{a} \,.
\end{equation}
and $C'(p)^{1/2p}\leq c'(p)$ holds with
\begin{equation}\label{twinkle4}
c'(p)=p\sqrt{d}\bigg(2^{p+1/2}\tilde{a}^{\frac{1}{2p}-\frac{1}{2}}+ 24\bigg).
\end{equation}
\end{lemma}
\begin{proof}
Recall equation \eqref{aver} and define
\begin{equation*}
\Delta_{n}:=\overline{\theta}^{\lambda}_n -\theta^*-\lambda\big(
h(\overline{\theta}^{\lambda}_n) - h(\theta^*)\big), \qquad \mbox{ for every } n\ge0.
\end{equation*}
Then, one calculates
\begin{align*}
 &\CPE{\|\overline{\theta}^{\lambda}_{n+1} - \theta^*\|^{2p}}{\overline{\theta}^{\lambda}_n}  = \CPE{\|\Delta_{n} + \sqrt{2\lambda}\xi_{n+1}\|^{2p}}{\overline{\theta}^{\lambda}_n} \\
 &\quad   = \CPE{\big(\|\Delta_{n}\|^2 + 2 \ps{\Delta_{n}}{\sqrt{2\lambda}\xi_{n+1}}+ \|\sqrt{2\lambda}\xi_{n+1}\|^2\big)^{p}}{\overline{\theta}^{\lambda}_n}   \\
 & \quad \le   \CPE{\sum_{\substack{i+j+k=p \\ \{i \leq  p-1\}\cap\{j\neq1\}}}\frac{p!}{i!j!k!}\|\Delta_{n}\|^{2i}\big(2\ps{\Delta_{n}}{\sqrt{2\lambda}\xi_{n+1}}\big)^j \|\sqrt{2\lambda}\xi_{n+1}\|^{2k}}{\overline{\theta}^{\lambda}_n}   \\
 & \quad +  \CPE{2p\|\Delta_{n}\|^{2(p-1)}\ps{\Delta_{n}}{\sqrt{2\lambda}\xi_{n+1}} }{\overline{\theta}^{\lambda}_n}
\end{align*}
where the last term is clearly zero. Thus, due to Lemma \ref{trivial},
\begin{align}\label{2p_norm_aver}
& \CPE{\|\overline{\theta}^{\lambda}_{n+1} - \theta^*\|^{2p}}{\overline{\theta}^{\lambda}_n}
   \le \CPE{\sum_{\substack{k=0 \\ k\neq 1}}^{2p}\binom{2p}{k}\|\Delta_{n}\|^{2p-k}\|\sqrt{2\lambda}\xi_{n+1}\|^{k}}{\overline{\theta}^{\lambda}_n}  \nonumber\\
&\quad   \le  \|\Delta_{n}\|^{2p} + \CPE{\sum_{k=2}^{2p}\binom{2p}{k}\|\Delta_{n}\|^{2p-k}\|\sqrt{2\lambda}\xi_{n+1}\|^{k}}{\overline{\theta}^{\lambda}_n}  \nonumber\\
&\quad    = \|\Delta_{n}\|^{2p} + \CPE{\Bigg(\sum_{k=2}^{2p}\binom{2p}{k}\|\Delta_{n}\|^{2p-k}\|\sqrt{2\lambda}\xi_{n+1}\|^{k-2}\Bigg) \|\sqrt{2\lambda}\xi_{n+1}\|^{2}}{\overline{\theta}^{\lambda}_n}  \nonumber\\
&\quad    = \|\Delta_{n}\|^{2p} + \CPE{\Bigg(\sum_{l=0}^{2(p-1)}\binom{2p}{l+2}\|\Delta_{n}\|^{2(p-1)-l}\|\sqrt{2\lambda}\xi_{n+1}\|^{l}\Bigg) \|\sqrt{2\lambda}\xi_{n+1}\|^{2}}{\overline{\theta}^{\lambda}_n}  \nonumber\\
&\quad    \le  \|\Delta_{n}\|^{2p} + \CPE{\binom{2p}{2} \Bigg(\sum_{l=0}^{2(p-1)}\binom{2(p-1)}{l}\|\Delta_{n}\|^{2(p-1)-l}\|\sqrt{2\lambda}\xi_{n+1}\|^{l}\Bigg) \|\sqrt{2\lambda}\xi_{n+1}\|^{2}}{\overline{\theta}^{\lambda}_n}  \nonumber\\
&\quad     = \|\Delta_{n}\|^{2p} + (2p-1)p \CPE{\big(\|\Delta_{n}\| +  \|\sqrt{2\lambda}\xi_{n+1}\|\big)^{2(p-1)} \|\sqrt{2\lambda}\xi_{n+1}\|^{2}}{\overline{\theta}^{\lambda}_n}   \nonumber\\
&\quad     \le  \|\Delta_{n}\|^{2p} + (2p-1)p 2^{2(p-1)}\|\Delta_{n}\|^{2(p-1)} \E[\|\sqrt{2\lambda}\xi_{n+1}\|^{2}]  \nonumber\\ &\qquad +  (2p-1)p 2^{2(p-1)} \E[\|\sqrt{2\lambda}\xi_{1}\|^{2p}].
\end{align}
Moreover, one recalls that for $\lambda < 2/(a+L_1)$
\[
\|\Delta_n\|^2 \le (1-2\tilde{a}\lambda)\|\overline{\theta}^{\lambda}_n -\theta^*\|^2.
\]
Consequently
\begin{align}\label{2p_norm_aver_2}
& \CPE{\|\overline{\theta}^{\lambda}_{n+1} - \theta^*\|^{2p}}{\overline{\theta}^{\lambda}_n} \nonumber\\
& \quad   \le  (1-2\tilde{a}\lambda)^{p}\|\overline{\theta}^{\lambda}_n -\theta^*\|^{2p} + (2p-1)p 2^{2p-1}\lambda d(1-2\tilde{a}\lambda)^{p-1}\|\overline{\theta}^{\lambda}_n -\theta^*\|^{2(p-1)} \nonumber\\
&  \qquad +  (2p-1)p 2^{2(p-1)} \E[\|\sqrt{2\lambda}\xi_{1}\|^{2p}]  \nonumber\\
& \quad   \le   (1-\tilde{a}\lambda)(1-2\tilde{a}\lambda)^{p-1}\|\overline{\theta}^{\lambda}_n -\theta^*\|^{2p} -\tilde{a}\lambda(1-2\tilde{a}\lambda)^{p-1}\|\overline{\theta}^{\lambda}_n -\theta^*\|^{2p} \nonumber\\
     & \qquad  + (2p-1)p 2^{2p-1}\lambda d(1-2\tilde{a}\lambda)^{p-1}\|\overline{\theta}^{\lambda}_n -\theta^*\|^{2(p-1)}  \nonumber\\
     & \qquad  +  (2p-1)p 2^{2(p-1)} \E[\|\sqrt{2\lambda}\xi_{1}\|^{2p}].
\end{align}
As a result, for $\|\overline{\theta}^{\lambda}_n -\theta^*\| \ge \overline{M}$, where $\overline{M} = \sqrt{d(2p-1)p2^{2p-1}/\tilde{a}}$, one obtains
\[
\CPE{\|\overline{\theta}^{\lambda}_{n+1} - \theta^*\|^{2p}}{\overline{\theta}^{\lambda}_n}
   \le  (1-\tilde{a}\lambda)\|\overline{\theta}^{\lambda}_n -\theta^*\|^{2p} +  \lambda(2p-1)p 2^{3p-2} \E[\|\xi_{1}\|^{2p}],
\]
whereas, for $\|\overline{\theta}^{\lambda}_n -\theta^*\| \le \overline{M}$ one obtains
\begin{align*}
\CPE{\|\overline{\theta}^{\lambda}_{n+1} - \theta^*\|^{2p}}{\overline{\theta}^{\lambda}_n}
   \le &  (1-\tilde{a}\lambda)\|\overline{\theta}^{\lambda}_n -\theta^*\|^{2p} + \lambda d^p(2p-1)^pp^p2^{p(2p-1)}\tilde{a}^{1-p}
   \\ & +  \lambda(2p-1)p 2^{3p-2} \E[\|\xi_{1}\|^{2p}]
\end{align*}
which yields \eqref{ly_aver}. Consequently, by Lemma \ref{miyazaki} below,
\begin{align*}
  \CPE{\|\overline{\theta}^{\lambda}_{n+1} - \theta^*\|^{2p}}{\overline{\theta}^{\lambda}_n}
   \le & (1-\tilde{a}\lambda)^{2p}\|\overline{\theta}_0 -\theta^*\|^{2p} +\frac{C'(p)}{\tilde{a}}.
\end{align*}
Thus, one obtains the desired result regarding the uniform bounds. The estimate $C'(p)^{1/2p}\leq c'(p)$ follows, noting
the trivial inequalities: $p^{1/p}\leq 2$, $p\in\nset\setminus\{0\}$; $(x+y)^{1/2p}\leq x^{1/2p}+y^{1/2p}$, $x,y\geq 0$.
\end{proof}

\subsection{Analysis for the SGLD scheme}\label{sec_nabdep}

%hence also
%\begin{equation}\label{nautilus}
%E\left\langle x, H(x,X_0) \right\rangle \ge a \|x\|^2  + \left\langle x, EH(0,X_0) \right\rangle.
%\end{equation}
One notes initially that the process in \eqref{aver} is Markovian while the one in (\ref{nab}) is not. %Furthermore, the process $X_t$, in general, satisfies weaker moment
%conditions than a Gaussian random variable so one cannot expect to find exponential
%Lyapunov functions as in Lemmas \ref{lem_foster}, \ref{lavel}.
However, uniform bounds are obtained in Lemma \ref{lem:lp-bound-theta}, below, for the  $2p$-th  moment of the SGLD scheme \eqref{nab}, for any $p\geq 1$. This result complements the findings of Lemma \ref{lem_foster} and is used in the proof of Theorem \ref{main}, which examines the convergence between the two sampling algorithms, ULA \eqref{aver} and SGLD \eqref{nab}, in Wasserstein-2 distance.

The following inequalities, derived from Assumptions \ref{assump:lip} and \ref{diss}, are often used:
\begin{equation}
\label{es6}
\begin{split}
\|H(\theta,x)\| &\le L_1\|\theta - \theta^*\| + L_2\|x\| + H^*, \quad H^* = \sum_{i=1}^d |H^i(\theta^*,0)|, \\
\ps{\theta - \theta^* }{H(\theta,x)} &\ge a \|\theta - \theta^*\|^2  + \ps{\theta-\theta^*}{H(\theta^*,x)} .
\end{split}
\end{equation}

\begin{lemma}\label{lem:lp-bound-theta}
Let Assumptions \ref{assump:lll}, \ref{assump:lip} and \ref{diss} hold. Let $V_p(\theta) = \|\theta-\theta^*\|^{2p}$ for some integer $p \ge 1$. The process $\theta^{\lambda}$ satisfies, for any $n \in \nset$ and $\lambda <\bar{\lambda}$, where $\bar{\lambda}$ is defined in \eqref{eq:definition-bar-lambda},%$\alpha=\rme^{-a/3}$,
	\begin{equation}\label{ly}
\E[V_p(\theta^{\lambda}_n)] \leq (\rho_{\lambda})^n\E[V_p(\theta^{\lambda}_0)] + \lambda C''(p),
	\end{equation}
	where $\rho_{\lambda} = 1-\tilde{a}\lambda \in (0,1)$ and
\begin{eqnarray*}
C''(p) &:=& (2^{2p}dp(2p-1))^{p}(2/\tilde{a})^{p-1}+2^{5p-4}p(2p-1) 2^{2p}d^p p^{\frac{3}{2}p}\\
&+& 2^{2p-1}\left\{(2p)^{2p}(2/\tilde{a})^{2p-1}+ (2^{2p-1}p(2p-1))^{p}(2/\tilde{a})^{p-1} \right.\\
&+& \left.  2^{4p-4}p(2p-1)\right\}\left\{ 2^{2p-1} L_1^{2p}\|\theta^*\|^{2p}+ 2^{2p-1}L_2^{2p} \mathcal{M}_{2p}(X)+ \{H^* \}^{2p}\right\}.
\end{eqnarray*}
%$g(X_0)$ is given explicitly in the proof. %the constant $C$ depends on $p,L, a$ and is of order $d^p$.
As a result,
\begin{equation}\label{twinkle}
\sup_{\lambda <\bar{\lambda}}\sup_{n} \E[V_p(\theta^{\lambda}_n)]\leq  \E[V_p(\theta_0)]+\frac{C''(p)}{\tilde{a}}.
\end{equation}
It follows also that $C''(p)^{1/2p}\leq c''(p)$ where
\begin{eqnarray}\nonumber
c''(p) &:=& p\sqrt{d}\bigg(2^{p+1/2}\tilde{a}^{\frac{1}{2p}-\frac{1}{2}}+ 48\bigg)\\
\nonumber &+& 2\left\{4p/\tilde{a}^{1-1/2p}+ 2^{p}p\sqrt{2}(2/\tilde{a})^{1/2-1/(2p)} \right.\\
\label{twinkle3} &+& \left.  12\right\}\left\{ 2L_1\|\theta^*\|+ 2L_2 \mathcal{M}_{2p}^{1/2p}(X)+ H^*\right\}.
\end{eqnarray}
\end{lemma}
\begin{proof}
%In this proof, we denote by $\theta_{t/\lambda}^{\lambda}$ = $\theta_t^{\lambda}$, for any $t \in [0,\lambda)$.
For each $n \in \nset$, denote by $\Delta_n = \theta_n^{\lambda} -  \theta^* -\lambda (H(\theta_n^{\lambda}, X_{n+1})-H(\theta^*, X_{n+1}))$. By direct calculations, one obtains,
\begin{align*}
&\CPE{\|\theta_{n+1}^{\lambda}- \theta^* \|^{2p}}{\theta_n^{\lambda}}\\
&= \CPE{\|\Delta_n+\sqrt{2\lambda}\xi_{n+1} -\lambda H(\theta^*, X_{n+1}) \|^{2p}}{\theta_n^{\lambda}}\nonumber	\\
		& =  \CPE{\left(\|\Delta_n\|^2+\|\sqrt{2\lambda}\xi_{n+1} -\lambda H(\theta^*, X_{n+1})\|^2 \right.\right.\right. \\
		&\hspace{8em} \left.\left.\left. +2\langle \Delta_n, \sqrt{2\lambda}\xi_{n+1} -\lambda H(\theta^*, X_{n+1})  \rangle \right)^{p}}{\theta_n^{\lambda}}\\
		& = \CPE{\sum_{k_1+k_2+k_3=p}\frac{p!}{k_1!k_2!k_3!}\|\Delta_n\|^{2k_1}\|\sqrt{2\lambda}\xi_{n+1} -\lambda H(\theta^*, X_{n+1}) \|^{2k_2}\right.\right. \\
		&\hspace{8em} \left.\left. \times(2\langle \Delta_n, \sqrt{2\lambda}\xi_{n+1} -\lambda H(\theta^*, X_{n+1})  \rangle )^{k_3}}{\theta_n^{\lambda}} \\
		&\leq \CPE{\|\Delta_n\|^{2p}}{\theta_n^{\lambda}} +2p\CPE{\|\Delta_n\|^{2p-2}\langle\Delta_n, \sqrt{2\lambda}\xi_{n+1} -\lambda H(\theta^*, X_{n+1}) \rangle }{\theta_n^{\lambda}} \\
		&\hspace{1em} +\sum_{k=2}^{2p}\binom{2p}{k}\CPE{\|\Delta_n\|^{2p-k}\|\sqrt{2\lambda}\xi_{n+1} -\lambda H(\theta^*, X_{n+1}) \|^k}{\theta_n^{\lambda}}.
\end{align*}
where the last inequality holds due to Lemma \ref{trivial}, and further calculations yield
\begin{align} \label{sgldbdeq2}
&\CPE{\|\theta_{n+1}^{\lambda}- \theta^* \|^{2p}}{\theta_n^{\lambda}}	\nonumber\\
		&\leq  \CPE{\|\Delta_n\|^{2p}}{\theta_n^{\lambda}} +2p\lambda \CPE{ \|\Delta_n\|^{2p-1} \| H(\theta^*, X_{n+1}) \| }{\theta_n^{\lambda}} \nonumber\\
		&\hspace{1em} +\sum_{k=2}^{2p}\binom{2p}{k}\CPE{\|\Delta_n\|^{2p-k}\|\sqrt{2\lambda}\xi_{n+1} -\lambda H(\theta^*, X_{n+1}) \|^k}{\theta_n^{\lambda}} \nonumber\\
		& \leq \left(1+\frac{\tilde{a}\lambda}{2}\right) \CPE{\|\Delta_n\|^{2p}}{\theta_n^{\lambda}}  +\lambda (2p)^{2p}\left(\frac{2}{\tilde{a}}\right)^{2p-1}  \CPE{  \| H(\theta^*, X_{n+1}) \|^{2p} }{\theta_n^{\lambda}} \nonumber\\
		&\hspace{1em} +2^{2p-3}p(2p-1)\CPE{\|\Delta_n\|^{2p-2} \|\sqrt{2\lambda}\xi_{n+1} -\lambda H(\theta^*, X_{n+1}) \|^2}{\theta_n^{\lambda}} \nonumber\\
		&\hspace{1em} +2^{2p-3}p(2p-1) \CPE{\|\sqrt{2\lambda}\xi_{n+1} -\lambda H(\theta^*, X_{n+1}) \|^{2p}}{\theta_n^{\lambda}} \nonumber\\
		&\leq (1+\tilde{a}\lambda) \CPE{\|\Delta_n\|^{2p}}{\theta_n^{\lambda}}  + \lambda (2p)^{2p}\left(\frac{2}{\tilde{a}}\right)^{2p-1}\CPE{  \| H(\theta^*, X_{n+1}) \|^{2p} }{\theta_n^{\lambda}}\nonumber\\
		&\hspace{1em} + \lambda (2^{2p-1}p(2p-1))^{p}\left(\frac{2}{\tilde{a}}\right)^{p-1}\CPE{  \| H(\theta^*, X_{n+1}) \|^{2p} }{\theta_n^{\lambda}}\nonumber\\
		&\hspace{1em} +  \lambda (2^{2p}dp(2p-1))^{p}\left(\frac{2}{\tilde{a}}\right)^{p-1}+\lambda 2^{5p-4}p(2p-1) \E[\|\xi_{n+1} \|^{2p}] \nonumber\\
		&\hspace{1em} + \lambda 2^{4p-4}p(2p-1)\CPE{  \| H(\theta^*, X_{n+1}) \|^{2p} }{\theta_n^{\lambda}},
\end{align}
where the second inequality follows the same argument as in the proof of Lemma \ref{lem_foster}.
Moreover, for $\lambda < 2/(a+L_1)$,
\begin{align*}
	 \CPE{\|\Delta_n\|^{2p}}{\theta_n^{\lambda}} 	& = \CPE{ \bigg(\|\theta_n^{\lambda}-\theta^*\|^2 -2\lambda \langle \theta_n^{\lambda}-\theta^*, H(\theta_n^{\lambda}, X_{n+1})-H(\theta^*, X_{n+1})\rangle\right.\right. \\
						&\hspace{3em}\left.\left. +\lambda^2\|H(\theta_n^{\lambda}, X_{n+1})-H(\theta^*, X_{n+1})\|^2\bigg)^p}{\theta_n^{\lambda}}\\
	& \leq (1-2\tilde{a}\lambda)^p\|\theta_n^{\lambda}-\theta^*\|^{2p}.
\end{align*}
Then, substituting the above estimate into \eqref{sgldbdeq2} yields
\[
\CPE{\|\theta_{n+1}^{\lambda}- \theta^* \|^{2p}}{\theta_n^{\lambda}} \leq (1-\tilde{a}\lambda)\|\theta_n^{\lambda}-\theta^*\|^{2p}+\lambda \CPE{g(X_{n+1})}{\theta_n^{\lambda}},
\]
where
\begin{align*}
g(X_{n+1}) & =
 (2^{2p}dp(2p-1))^{p}(2/\tilde{a})^{p-1}+2^{5p-4}p(2p-1) \E[\|\xi_{n+1} \|^{2p}]\\
 &\quad+2^{2p-1}\left\{(2p)^{2p}(2/\tilde{a})^{2p-1}+ (2^{2p-1}p(2p-1))^{p}(2/\tilde{a})^{p-1} \right.\\
& \quad \left.+ 2^{4p-4}p(2p-1)\right\}\left\{(L_1\|\theta^*\|+L_2\|X_{n+1}\|)^{2p}+\|H(\theta^*,0)\|^{2p}\right\}.
 \end{align*}
Using the trivial $(x+y)^{2p}\leq 2^{2p-1}(x^{2p}+y^{2p})$, $x,y\geq 0$ and Lemma \ref{miyazaki}, we
have
\[
E[g(X_{n+1})]\leq C''(p).
\]

Finally, denote by $\rho_{\lambda} = 1-\tilde{a}\lambda \in (0,1)$, then by induction, one obtains
\[
\E [\|\theta_{n+1}^{\lambda}- \theta^* \|^{2p}]  \leq (\rho_{\lambda})^{n+1}\E[\|\theta_0- \theta^*\|^{2p}] +\frac{C''(p)}{\tilde{a}} ,
\]
which implies $\sup_{\lambda <\bar{\lambda}}\sup_n \E [\|\theta_{n+1}^{\lambda}- \theta^* \|^{2p}] \leq \E [\|\theta_0- \theta^*\|^{2p}] +C''(p)/\tilde{a} $.
It is easy to check
%$E[g(X_0)]\leq C''(p)$ and
$C''(p)^{1/2p}\leq c''(p)$, too.
\end{proof}

Uniform $L^2$ bounds for the process in (\ref{nab}) are obtained in \cite{raginsky:rakhlin:telgarsky:2017} under dissipativity condition on $\nabla U$ and the $L^2$ error of the stochastic gradient, i.e. $\E[\|H(\theta, X_n) - h(\theta)\|^2]$, see their Assumptions $(\textbf{A.3}), (\textbf{A.4})$. In that paper a large size mini-batch could be used to reduce the variance of the estimator, which requires more computational costs. We could also incorporate mini-batches in our algorithm but this is not pursued here. For stability, the variance of the estimator has to be controlled, see \cite{teh}.

\subsection{Proof of Theorem \ref{main}}\label{sec_proof_main}

We now sketch a roadmap for the proof of Theorem \ref{main}. The time axis is cut into intervals of size $T$.
An auxiliary process $\overline{z}^{\lambda}$ is introduced which equals $\theta^{\lambda}$ at the points $nT$, $n\in\nset$ but it follows the averaged dynamics on $[nT,(n+1)T)$, see \eqref{aver}.

Using the conditional $L$-mixing property, one obtains estimates for the $L^2$-distance between $\overline{z}^{\lambda}$ and $\theta^{\lambda}$.
If $\overline{z}^{\lambda}$ were uniformly bounded, these would be of the order $\sqrt{\lambda}$. However, $\overline{z}^{\lambda}$ is unbounded
hence its maximal process needs to be controlled which leads to the somewhat weaker rate $\lambda^{\frac{1}{2}-\varepsilon}$,
for $\varepsilon>0$ arbitrarily small.

Next, estimates for the difference between $\overline{z}^{\lambda}$ and $\overline{\theta}^{\lambda}$ are derived
using the contraction property of the dynamics of $\overline{\theta}^{\lambda}$, see Lemma \ref{lem:properties-R-lambda}. It follows that this is of the same order as
$\overline{z}^{\lambda}-\theta^{\lambda}$. These observations then allows us to conclude.

We proceed now with the rigorous arguments. Let
\[
\mathcal{H}_n:=\mathcal{G}_n\vee\sigma(\xi_j,\ j\in\nset),\quad
\mathcal{H}^+_n:=\mathcal{G}^+_n,\ n\in\nset.
\]
Observe first that, since
$(X_n)_{n \in \nset}$ is conditionally $L$-mixing with respect to
$(\mathcal{G}_n,\mathcal{G}^+_n)_{n \in \nset}$, it is
conditionally $L$-mixing with respect to
$(\mathcal{H}_n,\mathcal{H}^+_n)_{n \in \nset}$, too, and the corresponding quantities ($M$, $\Gamma$, $\mathcal{C}$, $\mathcal{M}$)
remain the same.

For each $\theta\in\mathbb{R}^{d}$, $0\leq i\leq j$, one recursively defines
\[
z^{\lambda}(i,i,\theta):=\theta,\quad z^{\lambda}(j+1,i,\theta):=z^{\lambda}(j,i,\theta)
-\lambda h(z^\lambda(j,i,\theta))+\sqrt{2\lambda}\xi_{j+1}.
\]
Let $T:=\lfloor 1/\lambda\rfloor$, then for each $n\in\nset$ and for each
$nT\leq k<(n+1)T$, one defines
\[
\overline{z}_k^{\lambda}:=z^{\lambda}(k,nT,\theta^{\lambda}_{nT}).
\]
Consequently, $\overline{z}^{\lambda}_k$ is defined for all $k\in\nset$; $\overline{z}^{\lambda}_{nT}=\theta^{\lambda}_{nT}$
for $n\in\nset$ and $\overline{\theta}^{\lambda}_k=z^{\lambda}(k,0,\theta_0)$. Next, some simple but essential moment estimates are derived.

\begin{lemma}\label{adf} Let $q\geq 1$ be an integer. Then, for all $\lambda <\bar{\lambda}$, where $\bar{\lambda}$ is defined in \eqref{eq:definition-bar-lambda},
\[
\sup_{k\in\nset} \|\overline{z}_{k}^{\lambda} - \theta^* \|_{2q}\leq \underline{C}(q)
\]
holds for
\begin{equation}\label{constant_z_bar}
\underline{C}(q):=\|\theta_{0} - \theta^*\|_{2q} + \frac{c'(q)+c''(q)}{\tilde{a}^{1/(2q)}},
\end{equation}
where $c'(q)$, $c''(q)$ are as in Lemmata \ref{lem_foster} and \ref{lem:lp-bound-theta}.
\end{lemma}
\begin{proof}
Define $V_q(\theta):=\| \theta-\theta^*\|^{2q}$, $\theta\in\mathbb{R}^d$.
Let $k\in\nset$ be arbitrary and let $n\in\nset$ be such that $nT\leq k<(n+1)T$.
Note that \eqref{twinkle2} and \eqref{twinkle} imply
\begin{align*}
\sup_{nT\leq k< (n+1)T} \|\overline{z}_{k}^{\lambda} - \theta^*\|_{2q}
&\leq \left[\E[V_q(\theta^{\lambda}_{nT})] + \frac{C'(q)}{\tilde{a}}\right]^{1/(2q)}\\
&\leq \|\theta_{0}-\theta^*\|_{2q} + \frac{C'(q)^{1/(2q)} + C''(q)^{1/(2q)}}{\tilde{a}^{1/(2q)}},
\end{align*}
\end{proof}

\begin{lemma}\label{easy}
For all $\lambda <\bar{\lambda}$, where $\bar{\lambda}$ is defined in \eqref{eq:definition-bar-lambda}, it holds that
\[
\sup_{n\in\nset}\left[ \|H(\theta^{\lambda}_n,X_{n+1})\|_2+ \|h(\overline{z}_n^{\lambda})\|_2 \right]\leq C^{\flat},
\]
where
\begin{equation}\label{beethoven}
C^{\flat}=L_1\left[\|\theta_0-\theta^*\|_2+\frac{C''(1)}{\tilde{a}}\right] + 2L_2\mathcal{M}^{1/2}_{2}(X)+2 H^*
+\underline{C}(1)L_1.
\end{equation}
\end{lemma}
\begin{proof}
The first inequality of \eqref{es6} implies
\[
\|H(\theta^{\lambda}_n,X_{n+1})\|_2  \le L_1 \|\theta_n^{\lambda} - \theta^* \|_2 +L_2 \|X_n\|_2+ H^*.
\]
Combining this with Lemma \ref{lem:lp-bound-theta} (applied with $p=1$) shows that
\[
\sup_{\lambda<\overline{\lambda}}\sup_{n} \|H(\theta^{\lambda}_n,X_{n+1})\|_2
\leq L_1\left[ \E^{1/2}[V_1(\theta_0)]+\frac{C''(1)^{1/2}}{\tilde{a}^{1/2}}\right] +L_2\mathcal{M}^{1/2}_{2}(X)+ H^*.
\]
A similar argument can be applied to $h(\overline{z}_n^{\lambda})$, in view of \eqref{process_constant},
\begin{align*}
\|h(\overline{z}^{\lambda}_n)\|_2
&\leq L_1 \|\overline{z}^{\lambda}_n - \theta^* \|_2+ L_2\mathcal{M}^{1/2}_{2}(X)+ H^* \\
&\leq \underline{C}(1)L_1+L_2 \mathcal{M}_{2}^{1/2}(X)+ H^*,
\end{align*}
where $\underline{C}(1)$ is given by \eqref{constant_z_bar}.
\end{proof}

\begin{lemma}\label{below}
For each $j\in\nset$, the random field $H(\theta,X_n)$, $n\in\nset$, $\theta\in \ball{\theta^*}{j}$ satisfies
\begin{equation}
\label{eq:below-moment}
M^n_r(H(\theta,X),\ball{\theta^*}{j})\leq L_1 j + L_2 M_r^n(X)+ H^{*},
\end{equation}
\begin{equation}
\label{eq:below-mixing}
\Gamma^n_r(H(\theta,X),\ball{\theta^*}{j})\leq 2L_2\Gamma^n_r(X).
\end{equation}
\end{lemma}
\begin{proof} Let $\theta\in \ball{\theta^*}{j}$. The Minkowski's inequality imply
for $k\geq n$ and $i \in \{1,\dots,m\}$,
\[
\CPE[1/r]{|H^i(\theta,X_k)|^r}{\mathcal{H}_n} \leq L^i_1 j + L^i_2 \CPE[1/r]{\|X_k\|^r}{\mathcal{H}_n}+ |H^i(\theta^*,0)| \,.
\]
Hence, using $\|X_k\| \leq \sum_{j=1}^m |X_k^j|$ and the Minkowski's inequality again, we obtain
\[
	M^n_r(H(\theta,X),\ball{\theta^*}{j},i)\leq L^i_1 j + L^i_2 M_r^n(X) + |H^i(\theta^*,0)|.
\]
Summing the above relation over the indices $i \in \{1,\dots, m\}$ we get \eqref{eq:below-moment}.
One also notes that, due to Lemma \ref{lem:useful-conditional-expectation},
\begin{align*}
&\CPE[1/r]{|H^i(\theta,X_k)-\CPE{H^i(\theta,X_k)}{\mathcal{H}_n\vee\mathcal{H}_{n-\tau}^+}|^r}{\mathcal{H}_n} \\
&\qquad \leq 2\CPE[1/r]{|H^i(\theta,X_k)-H^i(\theta,\CPE{X_k}{\mathcal{H}_n\vee\mathcal{H}_{n-\tau}^+})|^r}{\mathcal{H}_n} \\
&\qquad \leq 2L^i_2 \CPE[1/r]{\|X_k-\CPE{X_k}{\mathcal{H}_n\vee\mathcal{H}_{n-\tau}^+}\|^r }{\mathcal{H}_n} \leq 2L^i_2 \sum_{j=1}^{m}\gamma_r^n(X,\tau,j),
\end{align*}
which implies \eqref{eq:below-mixing}.
\end{proof}

%\begin{lemma}\label{maximal}
%	Let $(X_i)_{i \in \nset}$ be a sequence of random variables such that for some $p > 1$
%	\[M = \sup_{i \in \nset} \E\|X_i\|^p < \infty.\]
%	Then for $0 < r < p$,
%	\[ \E\left[ \sup_{1 \le i \le j}\|X_i\|^r \right] \le j^{r/p} M^{r/p}.\]
%\end{lemma}
%\begin{proof}
%One has
%	\[\E^{p/r}\left[ \sup_{1 \le i \le j}\|X_i\|^r \right] \le \E \left[ \sup_{1 \le i \le j}\|X_i\|^p \right]
%\leq  \E \left[ \sum_{i=1}^j \|X_i\|^p \right] \le jM,\]
%by Jensen's inequality.
%\end{proof}

We shall also need the following measure-theoretical lemma.

\begin{lemma}\label{lem_meas}
Let $k\geq nT$ be an integer. There exists a version $h_{k,nT}: \Omega \times \mathbb{R}^{d} \to \mathbb{R}^{d}$ of $\E[H(\theta,X_k)\vert\mathcal{H}_{nT}],\ \theta\in\mathbb{R}^{d}$ which is jointly measurable.
\end{lemma}
\begin{proof}
For a fixed $\theta \in \mathbb{R}^{d}$, the conditional expectation $\CPE{H(\theta,X_k)}{\mathcal{H}_{nT}},\ \theta\in\mathbb{R}^{d}$ is a $\mathcal{H}_{nT}$-measurable random variable. We will construct a function
$h_{k,nT}: \Omega \times \mathbb{R}^{d} \to \mathbb{R}^{d}$ that is measurable in its second variable and, for all $\theta\in\mathbb{R}^{d}$, $h_{k,nT}$ is a version of $\E[H(\theta,X_k)\vert\mathcal{H}_{nT}]$. The case $k=nT$ is trivial. Let $k>nT$. It is enough to construct $h_{k,nT}(\theta)$, $\theta\in \ball{\theta^*}{N}$
for each $N\in\nset$.
Consider $\mathbb{B}(N):=\mathbf{C}(\ball{\theta^*}{N};\mathbb{R}^d)$, the usual Banach space
of continuous, $\mathbb{R}^d$-valued functions defined on $\ball{\theta^*}{N}$, equipped
with the maximum norm. The function
\[
\omega\in\Omega\to G_N(\omega):=(H(\theta,X_{k}(\omega))_{\theta\in \ball{\theta^*}{N}}),\ \omega\in\Omega,
\]
is a $\mathbb{B}(N)$-valued random variable and, by \eqref{es6},
\[
\sup_{\theta\in \ball{\theta^*}{N}} \|H(\theta,X_{k})\|\leq L_1 N+ L_2 \|X_{k}\|+H^*,
\]
which clearly has finite expectation as the process $X_n$, $n\in\nset$ is conditionally $L$-mixing.
Moreover, \cite[Proposition V.2.5]{neveu} implies the existence of a $\mathbb{B}(N)$-valued
random variable $\mathfrak{G}_N$ such that, for each $\mathbf{b}$ in the dual space $\mathbb{B}'(N)$
of $\mathbb{B}(N)$, \[
\E[\mathbf{b}(G_N)\vert \mathcal{H}_{nT}]=\mathbf{b}(\mathfrak{G}_N).
\]
This implies, in particular, that for all $\theta\in \ball{\theta^*}{N}$,
$\E[H(\theta,X_{k})\vert \mathcal{H}_{nT}]=\mathfrak{G}_N(\theta)$.
We may thus set $h_{k,nT}(\omega,\theta):=\mathfrak{G}_N(\omega,\theta)$. Since $(\omega,\theta)\to
\mathfrak{G}_N(\omega,\theta)$ is measurable in its first variable and continuous in
the second, it is, in particular, jointly measurable, see e.g. \cite[Lemma 4.50]{ab}.
\end{proof}

\begin{lemma}\label{lem:kaaka}
Assume \ref{assump:lll} and \ref{assump:lll} and let $p \geq 1$.
\[
\sup_{n\in\nset}\E^{1/p}\left[\left(
\sum\nolimits_{k=nT}^{\infty} \sup\nolimits_{\theta\in\mathbb{R}^d}\left\Vert h_{k,nT}(\theta)- h(\theta)\right\Vert
\right)^{p}\right]\leq 2L_2 \mathcal{C}_{p,1}(X) ,
\]
where $\mathcal{C}_{p,1}(X)$ is defined in \eqref{eq:process_constant_cond}.
\end{lemma}
\begin{proof}
Let $k \geq nT$.  Notice that, since $X_k$ and $\mathcal{G}_{nT}^+$ are independent of $\sigma(\xi_j, j \in \nset)$,  $\CPE{X_k}{\mathcal{H}_{nT}^+}= \CPE{X_k}{\mathcal{G}_{nT}^+}$, $\P$-a.s.
Since $\mathcal{G}_{nT}^+$ and $\mathcal{G}_{nT}$ are independent, we get that, for all $k \geq nT$, $\P$-a.s.,
\[
\CPE{H(\theta,\CPE{X_k}{\mathcal{G}_{nT}^+})}{\mathcal{H}_{nT}}= \CPE{H(\theta,\CPE{X_k}{\mathcal{G}_{nT}^+})}{\mathcal{G}_{nT}}= \E[H(\theta, \E[X_k\vert\mathcal{G}_{nT}^+])] \,.
\]
This implies that, for all $k \geq nT$,
\begin{align*}
\|h_{k,nT}(\theta)-h(\theta)\| &\leq  \left\|\CPE{H(\theta,X_k)}{\mathcal{G}_{nT}}-
		\CPE{H(\theta,\CPE{X_k}{\mathcal{G}_{nT}^+})}{\mathcal{G}_{nT}} \right\| 	\\
&+ \left\|\E[H(\theta, \CPE{X_k}{\mathcal{G}_{nT}^+})]-\E[H(\theta,X_k)]\right\| \\
& \leq L_2\CPE{\|X_k-\CPE{X_k}{\mathcal{G}_{nT}^+}\|}{\mathcal{G}_{nT}}
		+L_2\E\left[\|X_k-\CPE{X_k}{\mathcal{G}_{nT}^+}\|\right].
\end{align*}	
Using the Minkowski inequality, we get
\begin{align*}
\E^{1/p}&\left[\sup\nolimits_{\theta\in\mathbb{R}^{d}}\|h_{k,nT}(\theta)-h(\theta)\|^{p}\right]\\
&\leq
L_{2}\E^{1/p}[\|X_k-\CPE{X_k}{\mathcal{G}_{nT}^+}\|^{p}] + L_{2}
\E\left[\|X_k-\CPE{X_k}{\mathcal{G}_{nT}^+}\|\right]\\
&\leq 2L_{2}\sum_{i=1}^{m}\gamma^{0}_{p}(X,k-nT,i),
\end{align*}
noting that $\mathcal{G}_{0}$ is the trivial sigma algebra. This concludes the proof since
$\E[\Gamma_p^{0}(X)]\leq \mathcal{C}_{p,1}(X)$.
\end{proof}

\begin{proof}[\bf{ Proof of Theorem \ref{main}}] Fix
$n\in\nset$ and let $nT\leq k<(n+1)T$ be an arbitrary integer. By the triangle inequality, the difference of $\theta^{\lambda}$
and $\overline{\theta}^{\lambda}$ is decomposed into two parts
\begin{equation}\label{decomposition}	\|\theta^{\lambda}_k-\overline{\theta}^{\lambda}_k\|\leq \|{\theta}^{\lambda}_k-\overline{z}_k^{\lambda}\|+
	\|\overline{z}_k^{\lambda}-\overline{\theta}^{\lambda}_k\|.\end{equation}
Let $h_{k,nT}$ be the functional constructed in Lemma \ref{lem_meas}. Then, one estimates
\begin{eqnarray*}
\| \theta_{k}^{\lambda}-\overline{z}^{\lambda}_{k}\| &\leq& \lambda \left\Vert\sum\nolimits_{i=nT}^{k-1} 	
\left(H(\theta^{\lambda}_i,X_i)-h(\overline{z}^{\lambda}_i)\right)\right\Vert \leq
		\lambda \sum\nolimits_{i=nT}^{k-1} \Vert
		H(\theta^{\lambda}_i,X_i)-H(\overline{z}^{\lambda}_i,X_i)\Vert \\
		&& +
		\lambda \left\Vert\sum\nolimits_{i=nT}^{k-1}
		\left(H(\overline{z}^{\lambda}_i,X_i)-h_{i,nT}(\overline{z}^{\lambda}_i)\right) \right\Vert +
		\lambda \sum\nolimits_{i=nT}^{k-1}
		\left\Vert h_{i,nT}(\overline{z}^{\lambda}_i)- h(\overline{z}^{\lambda}_i)\right\Vert \\
&\leq&	\lambda L_1 \sum_{i=nT}^{k-1} \|\theta^{\lambda}_i-\overline{z}^{\lambda}_i\| + 		\lambda
		\max_{nT\leq m< (n+1)T} \left\|\sum_{i=nT}^m \left(H(\overline{z}_i^{\lambda},X_i)-
		h_{i,nT}(\overline{z}^{\lambda}_i)\right)\right\|\\
		&&  + 		\lambda \sum_{i=nT}^{\infty} \Vert h_{i,nT}(\overline{z}_i^{\lambda})- h(\overline{z}_i^{\lambda})\Vert
	\end{eqnarray*}
	due to Assumption \ref{assump:lip}. Thanks to Lemmas~\ref{lem:lp-bound-theta}, \ref{adf}, \ref{easy}, and \ref{lem:kaaka} all the terms
	on the RHS of the previous inequality are almost surely finite.
A discrete-time version of Gr\"onwall's lemma and taking squares lead to
	\begin{align*}
		\|\theta^{\lambda}_{k}-\overline{z}_{k}^{\lambda}\|^2 \leq
		2\lambda^2 \rme^{2L_1T\lambda}
		&\left[\max_{nT\leq m< (n+1)T} \left\|\sum_{i=nT}^{m} \left(H(\overline{z}_i^{\lambda},X_i)-
		h_{i,nT}(\overline{z}^{\lambda}_i)\right)\right\|^2 \right.\\
		  &+
		\left. \left(\sum_{i=nT}^{\infty} \left\Vert h_{i,nT}(\overline{z}_i^{\lambda})- h(\overline{z}_i^{\lambda})\right\Vert\right)^2\right],
	\end{align*}
	noting also $(x+y)^2\leq 2(x^2+y^2)$, $x,y\in\mathbb{R}$.
	Let us define the $\mathcal{H}_{nT}$-measurable random variable
\[
	N_{nT}:=\max_{nT\leq k< (n+1)T}\|\overline{z}_k^{\lambda} - \theta^*\|.
\]
	Now, by recalling the definition of $T$ and taking $\mathcal{H}_{nT}$-conditional expectations, one obtains
\begin{align*}
\E^{1/2}\left[\|\theta^{\lambda}_{k}-\overline{z}_{k}^{\lambda}\|^2 \vert\mathcal{H}_{nT}\right] \leq &  \sqrt{2}\lambda \rme^{L_1}\sum_{j=1}^{\infty} \indiacc{j-1\leq N_{nT} <j} \\
	 &\times
	\CPE[1/2]{\max_{nT\leq m<(n+1)T} \left\|\sum\nolimits_{i=nT}^m \left(H(\overline{z}_i^{\lambda},X_i)-
	h_{i,nT}(\overline{z}_i^{\lambda})\right)\right\|^2}{\mathcal{H}_{nT}} \\
	& +  \sqrt{2}\lambda \rme^{L_1}\sup_{\theta\in\mathbb{R}^d}
\sum_{i=nT}^{\infty} \left\Vert h_{i,nT}(\theta)- h(\theta)\right\Vert.
\end{align*}
Define for $n \in \nset$,
\begin{equation}\label{tildeZ}
\tilde{Z}_{n,k}^{\lambda}(j):=
\begin{cases}
H(\overline{z}_k^{\lambda},X_k)\indiacc{\|\overline{z}^{\lambda}_k - \theta^* \|\leq j}, &  nT\leq k<(n+1)T, \\
0 & \text{otherwise}
\end{cases}
\end{equation}
Recalling the $\mathcal{H}_{nT}$-measurability of $\overline{z}^{\lambda}_k$,
$nT\leq k<(n+1)T$, and arguing like in  Lemma~\ref{below}, one obtains
\begin{equation}\label{wqtilde}
\begin{split}
M^{nT}_r(\tilde{Z}_n^{\lambda}(j))&\leq L_1j +L_2 M_r^{nT}(X) + H^*\\
\Gamma^{nT}_r(\tilde{Z}_n^{\lambda}(j))&\leq 2L_2\Gamma^{nT}_r(X)
\end{split}
\end{equation}
With these notation, for each $j\in\nset$, the process defined by
\begin{equation}\label{Z}
Z_{n,k}^{\lambda}(j):=(H(\overline{z}_k^{\lambda},X_k)-h_{k,nT}(\overline{z}_k^{\lambda})) \indiacc{\|\overline{z}^{\lambda}_k - \theta^*\|\leq j}
=\tilde{Z}^{\lambda}_{n,k}(j)-\CPE{\tilde{Z}^{\lambda}_{n,k}(j)}{\mathcal{H}_{nT}},
\end{equation}
for $nT\leq k<(n+1)T$, $n\in\nset$  satisfies
\begin{equation}\label{wq}
\begin{split}
M^{nT}_r(Z_n^{\lambda}(j))&\leq 2[L_1 j + L_2 M_r^{nT}(X) + H^*] , \\
\Gamma^{nT}_r(Z_n^{\lambda}(j))&\leq 2L_2\Gamma^{nT}_r(X) \,.
\end{split}
\end{equation}

Notice that $Z^{\lambda}_{n,nT}(j)=0$ hence the maximum can be taken over $nT<m<(n+1)T$ instead of $nT\leq m<(n+1)T$.
One then applies Theorem \ref{harmadik} with the choice $n={nT}$, $r=3$, $b_i\equiv 1$, $X_k:=Z_{n,k}^{\lambda}(j)$ to obtain
\begin{equation}
\begin{split}
&\indiacc{N_{nT} \leq j}\CPE[1/2]{\max_{nT<m<(n+1)T} \left\|\sum\nolimits_{i=nT+1}^m \left(H(\overline{z}_i^{\lambda},X_i)-
	h_{i,nT}(\overline{z}_i^{\lambda})\right)\right\|^2}{\mathcal{H}_{nT}}\\
& \quad \leq \indiacc{N_{nT} \leq j}
\CPE[1/3]{\max_{nT< m<(n+1)T} \left\|\sum\nolimits_{i=nT+1}^m \left(H(\overline{z}_i^{\lambda},X_i)-
	h_{i,nT}(\overline{z}_i^{\lambda})\right)\right\|^3}{\mathcal{H}_{nT}}\\
& \quad \leq 10\indiacc{N_{nT}\leq j}\sqrt{T}[\Gamma_3^{nT}(Z_n^{\lambda}(j))+M_3^{nT}(Z_n^{\lambda}(j))],
\end{split}
\label{horus}
\end{equation}
noting that $C'(3)\leq 10$ holds for the constant $C'(3)$ appearing in Theorem \ref{harmadik}.

Now we turn to estimating $N_{nT}$. Let $q> 1$ be an arbitrary integer.
Let us apply Lemma \ref{lem:maximal} with the choice $r:=2$ and $p:=2q$ to obtain
\begin{equation}\label{kitsch}
\E[N_{nT}^2]\leq T^{2/(2q)}\sup_{nT\leq k< (n+1)T} \E^{2/(2q)}[\|\overline{z}_{k}^{\lambda} - \theta^*\|^{2q}],
\end{equation}
which implies, by Lemma \ref{adf},
\begin{equation}\label{misch}
\E[(N_{nT}+1)^2]\leq 2[1+ T^{2/(2q)}\underline{C}^2(q)].
\end{equation}

By the Cauchy-Schwarz inequality,
\eqref{wq} and \eqref{misch} we can
perform the auxiliary estimate
%an application of Theorem \ref{estim}
\begin{eqnarray}& & \label{morus} \sum_{j=1}^{\infty}\E[\indiacc{j-1\leq N_{nT} <j}[\Gamma_3^{nT}(Z_n^{\lambda}(j))+M_3^{nT}(Z_n^{\lambda}(j))]^{2}]\\
\nonumber &\leq  & 8\sum_{j=1}^{\infty}\E[\indiacc{j-1\leq N_{nT} <j}	[L_{2}^{2}(\Gamma^{nT}_3(X))^2+	[L_1j + L_2M_3^{nT}(X)+H^*]^2] \\
\nonumber &\leq& 8L_{2}^{2}\E[(\Gamma^{nT}_3(X))^2]+24[\E[L_{1}^{2}(N_{nT}+1)^{2} +L_{2}^{2} (M_3^{nT}(X))^{2} +(H^{*})^{2}]\\
\nonumber &\leq& 8L_{2}^{2}\mathcal{C}_{3,2}+24[L_{2}^{2}\mathcal{M}^{2/3}_{3} +(H^{*})^{2}]+48L_{1}^{2}[1+T^{2/2q}]\underline{C}^{2}(q)\\
\nonumber &\leq& 96 T^{2/p}[L_{1}^{2}\underline{C}^{2}(q)+L_{2}^{2}\mathcal{C}_{3,2}+L_{2}^{2}\mathcal{M}^{2/3}_{3} +(H^{*})^{2}],
\end{eqnarray}
using the notation introduced for conditionally $L$-mixing processes in \eqref{eq:process_constant_cond} and the trivial $T\geq 1$ (in the last
inequality).
%8\sum_{j=1}^{\infty} \P^{1/2}(N_{nT}+1\geq j)
%\sqrt{2}\E^{1/2}[\{\Gamma_3^{nT}(Z_n^{\lambda}(j))\}^2+\{M_3^{nT}(Z_n^{\lambda}(j))\}^2]\\
%\nonumber & \leq	
%\sum_{j=1}^{\infty}\sqrt{\frac{\E(N_{nT}+1)^6}{j^6}}4 L_2 \E^{1/2}[(\Gamma^{nT}_3(X))^2+  [L_1j + L_2M_3^{nT}(X)+H^*]^2] \\
%\nonumber & \leq	\sum_{j=1}^{\infty}\sqrt{\frac{\E(N_{nT}+1)^6}{j^4}}4 L_2
%\%E^{1/2}[(\Gamma^{nT}_3(X))^2+  [L_1 + L_2M_3^{nT}(X)+H^*]^2]\\
%\nonumber & \leq	\sum_{j=1}^{\infty} 24 L_2 [1+ T^{3/p}\underline{C}^3(p/2)]
%\frac{\mathcal{C}^{1/2}_{3,2}(X)+L_1 + L_2\mathcal{M}^{1/3}_{3}(X)+H^*}{j^2}\\
%& \leq& \sum_{j=1}^{\infty} 6\times 16\times L_2^2 L^2[1+ T^{3/p}\underline{C}(p)^3]
%\frac{\mathcal{C}_{4,3}(X)j[\mathcal{M}_{4,3}(X)+1+\|H(\theta^*,0)\|]}{j^3}\\
%\nonumber & \leq 24 \frac{\pi^2}{6}L_2 [1+ T^{3/p}\underline{C}^3(p/2)]
%[\mathcal{C}^{1/2}_{3,2}(X)+L_1 + L_2\mathcal{M}^{1/3}_{3}(X)+H^*]\\
%& \leq 8 \pi^2 L_2 T^{3/p}\underline{C}^3(p/2)
%[\mathcal{C}^{1/2}_{3,2}(X)+L_1 + L_2\mathcal{M}^{1/3}_{3}(X)+H^*],\label{morus}
We define
\[
C^{\sharp}(p):= 96[L_{1}^{2}\underline{C}^{2}(p/2)+L_{2}^{2}\mathcal{C}_{3,2}+L_{2}^{2}\mathcal{M}^{2/3}_{3} +(H^{*})^{2}]+4L_{2}^{2}\mathcal{C}_{2,1}^{2}.
\]
Notice that $(C^{\sharp})^{1/2}\leq C^{\star}$, where the latter constant is given by
\begin{equation}\label{C-star}
C^{\star}(p):=10[L_{1}\underline{C}(p/2)+L_{2}\mathcal{C}^{1/2}_{3,2}+L_{2}\mathcal{M}^{1/3}_{3} +H^{*}]+2L_{2}\mathcal{C}_{2,1}.
\end{equation}
We conclude from \eqref{horus}, \eqref{morus} and \eqref{C-star} that
\[
\E^{1/2}\| \theta^{\lambda}_{k}-\overline{z}_{k}^{\lambda}\|^2\leq 15e^{L_1} C^{\star}(p)[\lambda \sqrt{T}
T^{1/p}+\lambda]\leq 30 \rme^{L_1} C^{\star}(p)\lambda^{\frac{1}{2} - \frac{1}{p}},
\]
for all $k\in\nset$, noting also that $\sqrt{2}\leq 3/2$.

Now we turn to
estimating $\|\overline{z}_k^{\lambda}-\overline{\theta}^{\lambda}_k\|$ for $nT \le k < (n+1)T$. We compute
\begin{eqnarray*}
	\|\overline{z}_k^{\lambda}-\overline{\theta}^{\lambda}_k\|_2 &\leq &
	\sum_{i=1}^n \|z^{\lambda}(k,iT,\theta^{\lambda}_{iT}) - z^{\lambda}(k, (i-1)T, \theta^{\lambda}_{(i-1)T})\|_2  \\
	&=& \sum_{i=1}^n \|z^{\lambda}(k,iT,\theta^{\lambda}_{iT}) - z^{\lambda}(k, iT, z^{\lambda}(iT,(i-1)T, \theta^{\lambda}_{(i-1)T}))\|_2.
\end{eqnarray*}
By  Lemma~\ref{lem:properties-R-lambda}-\ref{item:contraction}, we estimate
\begin{align*}
&\|z^{\lambda}(k,iT,\theta^{\lambda}_{iT}) - z^{\lambda}(k, iT, z^{\lambda}(iT,(i-1)T, \theta^{\lambda}_{(i-1)T}))\|_2  \\
%&\leq& (1 - 2 \tilde{a} \lambda) \| z^{\lambda}(k-1,iT,\theta^{\lambda}_{iT}) - z^{\lambda}(k-1, iT, z^{\lambda}(iT,(i-1)T, \theta^{\lambda}_{(i-1)T}))  \|_2\\
&\leq   (1 - 2 \tilde{a} \lambda)^{k - iT} \|  \theta^{\lambda}_{iT} - z^{\lambda}(iT,(i-1)T, \theta^{\lambda}_{(i-1)T}) \|_2 \\
&\leq (1 - 2 \tilde{a} \lambda)^{k - iT} \|  \theta^{\lambda}_{iT-1} - \lambda H(\theta^{\lambda}_{iT-1}, X_{iT}) - \overline{z}^{\lambda}_{iT-1} + \lambda h(\overline{z}^{\lambda}_{iT-1}) \|_2\\
&\leq (1 - 2 \tilde{a} \lambda)^{k - iT}\left[ \| \theta^{\lambda}_{iT-1} - \overline{z}^{\lambda}_{iT-1}  \|_2 + \lambda \|H(\theta^{\lambda}_{iT-1}, X_{iT}) - h(\overline{z}^{\lambda}_{iT-1}) \|_2 \right]
\end{align*}
Using Lemma \ref{easy}, the estimation continues as follows
\begin{align*}
\|\overline{z}_k^{\lambda}-\overline{\theta}^{\lambda}_k\|_2
&\leq  \sum_{i = 1}^{n} \rme^{-2\tilde{a}\lambda(k-iT)}
	\left[\|\theta^{\lambda}_{iT -1} - \overline{z}^{\lambda}_{iT-1}\|_2 +
	\lambda \|H(\theta^{\lambda}_{iT-1},X_{iT}) - h(\overline{z}^{\lambda}_{iT-1})\|_2\right]\\
	\leq & \sum_{i = 1}^{n}\rme^{-2a\lambda(n-i)T}
[30\rme^{L_1} C^{\star}(p)\lambda^{\frac{1}{2} - \frac{1}{p}}+C^{\flat}\lambda]\\
&\leq \frac{30\rme^{L_1} C^{\star}(p)+C^{\flat}}{1-\rme^{-2\tilde{a}\lambda T}} \lambda^{\frac{1}{2} - \frac{1}{p}} \leq
\frac{30\rme^{L_1} C^{\star}(p)+C^{\flat}}{1-\rme^{-\tilde{a}}} \lambda^{\frac{1}{2} - \frac{1}{p}}.
\end{align*}
The proof is completed by setting
\begin{equation}\label{ccirc}
C_0(p):=\frac{30\rme^{L_1} C^{\star}(p)+C^{\flat}}{1-\rme^{-\tilde{a}}}+C^{\star}(p)
\end{equation}
and noting \eqref{decomposition}.
\end{proof}

\begin{remark} \label{dimension_dependence}
{\rm We track the dependence of the constant $C_{0}(p)$ (appearing in Theorem \ref{main}) on the dimension $d$.
Notice that Lemmata \ref{lem_foster} \ref{lem:lp-bound-theta} provide $c'(p)$ and $c''(p)$, both of which of the order
$\sqrt{d}$. This order is inherited by $\underline{C}(q)$ in Lemma \ref{adf} and thus results in $d^{1/2}$ in $C^{\star}(p)$ and $C^{\flat}$, see
\eqref{C-star} and \eqref{beethoven}. We finally get that $C_{0}(p)$ is of the order $d^{1/2}$.}
\end{remark}

\subsection{Proof of Theorem \ref{dros}}\label{sec_depdatacor}

To prove Theorem \ref{dros}, another convergence result is needed, which is the rate of convergence to stationarity of the recursive scheme \eqref{aver} in Wasserstein-$2$ distance. Note that with Lemma \ref{prop2} and \ref{lem:properties-R-lambda}, the convergence in Wasserstein-$2$ distance can be considered. The following is the adapted statement in \cite[Corollary 7]{durmus-moulines} using the notation of this article.
\begin{theorem} {\cite[Corollary 7]{durmus-moulines}} \label{thm6}
	Let Assumptions \ref{assump:lll}, \ref{assump:lip}, \ref{diss} hold and let $\lambda < \bar{\lambda}$ where $\bar{\lambda}$ is defined in \eqref{eq:definition-bar-lambda}. Then, the Markov chain $(\overline{\theta}^{\lambda}_n)_{n \in \nset}$ admits an invariant measure $\pi_{\lambda}$ such that, for all $n \in \nset$;
\[
W_2(\mathrm{Law}(\overline{\theta}^{\lambda}_n),\pi_{\lambda})\leq \hat{c} \rme^{-a\lambda n},\qquad n\in\nset,
\]
where $\hat{c}$ is coming from (iii) Lemma~\ref{lem:properties-R-lambda}:
\[
\hat{c}:=\sqrt{2}(\|\theta_0-\theta\|^2+d/\tilde{a})^{1/2}.
\]
%\erici{why are we repeating exactly item (iii) of Lemma~\ref{lem:properties-R-lambda}}
Furthermore,
	\begin{align*}
	W_2(\pi,\pi_{\lambda})\leq c\sqrt{\lambda},
	\end{align*}
where
\[
c=\left(L_1^2 \tilde{a}^{-1}(2\lambda+\tilde{a}^{-1})(d+\tfrac{1}{12}\lambda^2L_1^2d+\tfrac{1}{2}L_1^2\lambda d/a) \right)^{1/2}
\]
with $\tilde{a}$ defined in \eqref{B2-alt}.
\end{theorem}
Note that for the Langevin SDE \eqref{eq1}, the Euler and Milstein schemes coincide, which implies that the optimal rate of convergence for scheme \eqref{aver} is 1 instead of 1/2.
The bound provided in Theorem~\ref{thm6} can thus be improved under an additional smoothness assumption for the drift coefficient of \eqref{eq1}. However, as our main focus is the behaviour of the SGLD algorithm \eqref{nab} and, in view of Example \ref{example_rate_one_half}, it is known that its optimal rate of convergence is 1/2, any improvement on the behaviour of scheme \eqref{aver} does not change this fact.

\begin{proof}[\bf{Proof of Theorem \ref{dros}}]
Take $p$ large enough so that $\kappa>2/(p-1)$ and thus
$1/p\leq\kappa/(\kappa+2)$ holds. Denote by $\tilde{C} = \max\{C_0(p), \hat{c}, c\}$.
Theorems \ref{main} and \ref{thm6} imply that
\begin{align*}
W_2(\mathrm{Law}(\theta^{\lambda}_n),\pi)&\leq
W_2(\mathrm{Law}(\theta^{\lambda}_n),\mathrm{Law}(\overline{\theta}^{\lambda}_n))+
W_2(\mathrm{Law}(\overline{\theta}^{\lambda}_n),\pi_{\lambda})+
W_2(\pi_{\lambda},\pi) \\
&\leq \tilde{C}[\lambda^{\frac{1}{2}-\frac{3}{2p}}+\rme^{-a\lambda n}+\lambda^{\frac{1}{2}}] \\
&\leq 2\tilde{C}[\lambda^{\frac{1}{2+\kappa}}+\rme^{-a\lambda n}].
\end{align*}
%It is worth noting here that since $\tilde{C}= \max\{C_0(p), c_1, c\}$, its dependence on the dimension (which is denoted by $d$) is of order $\sqrt{d}$.
For $0< \epsilon <\mathrm{e}^{-1}$, choosing $\lambda:= \epsilon^{2+\kappa}/(4\tilde{C})^{2+\kappa}$,
$2\tilde{C}\lambda^{\frac{1}{2+\kappa}}\leq\epsilon/2$ holds. Now it remains
to choose $n$ large enough to have $\tilde{C}\rme^{-a\lambda n}\leq \epsilon/2$ or,
equivalently, $a\lambda n\geq \ln(2\tilde{C}/\epsilon)$. Noting the choice of $\lambda$
and $\ln(1/\epsilon)\geq 1$, this is possible if
\[
%n\geq\frac{\tilde{C}}{\epsilon^{2+\kappa}}\ln(1/\epsilon), \quad  $0< \epsilon <1/2$
n\geq\frac{c_2(\kappa)}{ \epsilon^{2+\kappa}}\ln(1/\epsilon),
\]
where $c_2(\kappa) = \frac{(4\tilde{C})^{2+\kappa}}{a }(1+\ln(2\tilde{C}))$.
\end{proof}

\section{Proof of main results: independent data} \label{sec_indepdata}
For the case of independent data, it is enough to obtain the second moment of the SGLD scheme \eqref{nab} before considering the convergence in Wasserstein-2 distance. The following lemma provides an upper bound  for the second moment of the scheme \eqref{nab} with explicit constants.%Under assumptions \ref{iidlocLip} - \ref{iidconv},

\begin{lemma}\label{indepbd}
Let Assumptions \ref{iidlocLip}, \ref{iid} and \ref{iidcnv} hold. Let
\begin{equation}\label{lambda_0}
\lambda_0 \eqdef \min \Big(a/2L_1^2\E[(1+\|X_0\|)^{2\rho}],\, 1/a \Big).
\end{equation}
For $\lambda \le \lambda_0$, the function $V_1(\theta):=\|\theta-\theta^*\|^2$ satisfies
\[
\CPE{V_1(\theta^{\lambda}_n)}{\theta_{n-1}^{\lambda}} \leq (1-a \lambda) V_1(\theta^{\lambda}_{n-1}) + \lambda C ,
\]
where
\[
C \eqdef 4L_2^2(1+\|\theta^*\|)^2 \E[(1+\|X_0\|)^{2\rho+2}] + 4\{ H^* \} ^2+ 2d \,.
\]
As a result, $\sup_{\lambda \leq \lambda_0} \sup_{n \in \nset} \E[V_1(\theta^{\lambda}_n)] < \infty$. Moreover, if $\rho=0$ in Assumption \ref{iidlocLip}, then the above result is true for $\lambda \le \min(1/2L_1, 1/(a+L_1))$.
\end{lemma}
\begin{proof}
By using the SGLD scheme \eqref{nab}, one calculates
\begin{align*}
 \|\theta_{n+1}^{\lambda} - \theta^*\|^2 &=  \|\theta_{n}^{\lambda}- \theta^*\|^2 + 2\langle\theta_n^{\lambda}- \theta^*, -\lambda H(\theta_n^{\lambda}, X_{n+1})+\sqrt{2\lambda}\xi_{n+1}\rangle \\
&\hspace{1em} + \| -\lambda H(\theta_n^{\lambda}, X_{n+1})+\sqrt{2\lambda}\xi_{n+1}\|^2 \\
&=  \|\theta_{n}^{\lambda}- \theta^*\|^2 - 2\lambda\langle\theta_n^{\lambda}- \theta^*, H(\theta_n^{\lambda}, X_{n+1}) - H(\theta^*, X_{n+1})\rangle   \\
 &\hspace{1em} + 2\langle\theta_n^{\lambda}- \theta^*,\sqrt{2\lambda}\xi_{n+1}\rangle -2\lambda\langle\theta_n^{\lambda}- \theta^*,   H(\theta^*, X_{n+1})\rangle   \\
  &\hspace{1em} + \lambda^2 \| H(\theta_n^{\lambda}, X_{n+1})\|^2- 2\lambda\langle H(\theta_n^{\lambda}, X_{n+1}),\sqrt{2\lambda}\xi_{n+1}\rangle +2\lambda\|\xi_{n+1}\|^2
\end{align*}
and thus
\begin{align}\label{no-co-coercive}
 &\E[\|\theta_{n+1}^{\lambda} - \theta^*\|^2 | \theta_{n}^{\lambda}]\nonumber \\
 &\le  \|\theta_{n}^{\lambda}- \theta^*\|^2 -2\lambda \E[\langle\theta_n^{\lambda}- \theta^*, A(X_{n+1})(\theta_n^{\lambda}- \theta^*)\rangle| \theta_{n}^{\lambda}] - 2\lambda \langle\theta_n^{\lambda}- \theta^*, h(\theta^*) \rangle \nonumber\\
 &\hspace{1em} + \lambda^2 \E[\| H(\theta_n^{\lambda}, X_{n+1})\|^2 | \theta_{n}^{\lambda}] + 2\lambda d  \\
&\le   \|\theta_{n}^{\lambda} - \theta^*\|^2 -2\lambda a\|\theta_{n}^{\lambda} - \theta^*\|^2 +  2\lambda^2 \E[\| H(\theta_n^{\lambda}, X_{n+1}) -  H(\theta^*, X_{n+1})\|^2| \theta_{n}^{\lambda}] \nonumber\\
 &\hspace{1em} +  2\lambda^2\E[\| H(\theta^*, X_{n+1})\|^2 ] + 2\lambda d. \nonumber
\end{align}
Hence, for $\lambda \le \min \Big(a/2L_1^2\E[(1+\|X_0\|)^{2\rho}],\, 1/a \Big)$
\begin{align*}
\E[\|\theta_{n+1}^{\lambda} - \theta^*\|^2 | \theta_{n}^{\lambda} ] \le&   (1-\lambda a)\|\theta_{n}^{\lambda} - \theta^*\|^2 + 4\lambda^2L_2^2(1+\|\theta^*\|)^2 \E[(1+\|X_0\|)^{2\rho+2}] \\ & + 4\lambda^2\{H^*\}^2+ 2\lambda d
\\ \Rightarrow \E(\|\theta_{n+1}^{\lambda}- \theta^*\|^2 | \theta_{n}^{\lambda}) \le & (1-\lambda a)\|\theta_{n}^{\lambda}- \theta^*\|^2 + \lambda C,
\end{align*}
where $C=4L_2^2(1+\|\theta^*\|)^2 \E[(1+\|X_0\|)^{2\rho+2}] + 4\{ H^* \}^2+ 2d$. Consequently, for any $n\ge 1$,
\begin{align*}
\E[\|\theta_{n}^{\lambda} - \theta^*\|^2] \le (1-\lambda a)^n \E[\|\theta_0 - \theta^*\|^2] + \frac{C}{ a}<\infty.
\end{align*}
Crucially, one observes here that if $\rho=0$ in Assumption \ref{iidlocLip}, then $H$ is co-coercive with the following property, for every $x\in \mathbb{R}^m$ and all $\theta, \theta^*\in \mathbb{R}^d$
\begin{equation}\label{co-coerc}
\langle\theta- \theta', H(\theta, x) - H(\theta', x)\rangle \ge \frac{1}{L_1} \|H(\theta, x) - H(\theta', x)\|^2.
\end{equation}
It follows that, in view of \eqref{co-coerc}, one rewrites \eqref{no-co-coercive} as follows
\begin{align*}
& \E[\|\theta_{n+1}^{\lambda} - \theta^*\|^2 | \theta_{n}^{\lambda}]\\
& \le  \|\theta_{n}^{\lambda}- \theta^*\|^2 -\lambda \E[\langle\theta_n^{\lambda}- \theta^*, A(X_{n+1})(\theta_n^{\lambda}- \theta^*)\rangle| \theta_{n}^{\lambda}]  \\
&\hspace{1em} - \frac{\lambda}{L_1} \|H(\theta_n^{\lambda}, X_{n+1}) - H(\theta^*, X_{n+1})\|^2+ 2\lambda \langle\theta_n^{\lambda}- \theta^*, h(\theta^*) \rangle \\
&\hspace{1em}+ \lambda^2 \E[\| H(\theta_n^{\lambda}, X_{n+1})\|^2 | \theta_{n}^{\lambda}] + 2\lambda d \\
 & \le  \|\theta_{n}^{\lambda} - \theta^*\|^2 -\lambda a\|\theta_{n}^{\lambda} - \theta^*\|^2 +  (2\lambda^2-\frac{\lambda}{L_1}) \E[\| H(\theta_n^{\lambda}, X_{n+1}) -  H(\theta^*, X_{n+1})\|^2 ] \\
 &\hspace{1em}  +  2\lambda^2\E[\| H(\theta^*, X_{n+1})\|^2 ] + 2\lambda d.
\end{align*}
which yields, for $\lambda \le 1/2L_1$
\begin{align*}
\E[\|\theta_{n+1}^{\lambda} - \theta^*\|^2 | \theta_{n}^{\lambda} ] \le &  (1-\lambda a)\|\theta_{n}^{\lambda} - \theta^*\|^2 + 4\lambda^2L_2^2(1+\|\theta^*\|)^2 \E[(1+\|X_0\|)^{2}] \\ & + 4\lambda^2 \{H^*\}^2+ 2\lambda d
\\ \Rightarrow \E(\|\theta_{n+1}^{\lambda}- \theta^*\|^2 | \theta_{n}^{\lambda}) \le & (1-\lambda a)\|\theta_{n}^{\lambda}- \theta^*\|^2 + \lambda C,
\end{align*}
where $C=4L_2^2(1+\|\theta^*\|)^2 \E[(1+\|X_0\|)^{2}] + 4 \{ H^* \}^2+ 2d$.

\end{proof}

\begin{proof}[\bf{Proof of Theorem \ref{iid-cor}}]
One notes that \textbf{(B1)} is still valid, with the only difference that the Lipschitz constant in \textbf{(B1)} is given by $L_1\E[(1+\|X_0\|)^{\rho}]$, and \textbf{(B2)} holds with $a$. Consequently, Theorem \ref{thm6} is still true.
The main steps of the proof of Theorem \ref{main} need to be reformulated for the \iid case. Initially, one notes that the following result holds due to Lemma \ref{indepbd}
\[
\sup_{\lambda \in (0,\lambda_0)} \sup_{n\ge0}\E[\|\theta_n^{\lambda} \|^2]<c_0,
\]
where $c_0 = 2\E\|\theta_0 - \theta^*\|^2 +2C/a + 2\|\theta^*\|^2$, and $C$ is given explicitly in Lemma \ref{indepbd}.
Then, using synchronous coupling for the schemes \eqref{aver} and \eqref{nab}, one obtains
\begin{align*}
&\|\theta_{n+1}^{\lambda} - \bar{\theta}_{n+1}^{\lambda}\|^2 \\
&=   \|\theta_{n}^{\lambda} - \bar{\theta}_{n}^{\lambda} -\lambda \Big( H(\theta_n^{\lambda}, X_{n+1}) - h(\bar{\theta}_{n}^{\lambda})\Big)\|^2 \\
&\le  \|\theta_n^{\lambda} - \bar{\theta}_n^{\lambda}\|^2 - 2\lambda\langle\theta_n^{\lambda} - \bar{\theta}_n^{\lambda}, H(\theta_n^{\lambda}, X_{n+1}) - h(\bar{\theta}_{n}^{\lambda})\rangle + \lambda^2 \|H(\theta_n^{\lambda}, X_{n+1}) - h(\bar{\theta}_{n}^{\lambda})\|^2 \\
\\
&\le  \|\theta_n^{\lambda} - \bar{\theta}_n^{\lambda}\|^2 - 2\lambda\langle\theta_n^{\lambda} - \bar{\theta}_n^{\lambda}, h(\theta_n^{\lambda}) - h(\bar{\theta}_{n}^{\lambda})\rangle - 2\lambda\langle\theta_n^{\lambda} - \bar{\theta}_n^{\lambda}, H(\theta_n^{\lambda}, X_{n+1}) - h(\theta_n^{\lambda})\rangle \\
& \hspace{1em}+ 2\lambda^2 \|H(\theta_n^{\lambda}, X_{n+1}) - h(\theta_{n}^{\lambda})\|^2 + 2\lambda^2 \| h(\theta_{n}^{\lambda}) - h(\bar{\theta}_{n}^{\lambda})\|^2.
 \end{align*}
 Taking expectations on both sides and using \eqref{B2-alt} yields
\begin{multline*}
\E[\|\theta_{n+1}^{\lambda} - \bar{\theta}_{n+1}^{\lambda}\|^2  |\theta_{n}^{\lambda},\bar{\theta}_{n}^{\lambda} ] \le \|\theta_n^{\lambda} - \bar{\theta}_n^{\lambda}\|^2 - 2\lambda \tilde{a}\|\theta_n^{\lambda} - \bar{\theta}_n^{\lambda}\|^2 -\frac{2\lambda}{a+L_1} \| h(\theta_{n}^{\lambda}) - h(\bar{\theta}_{n}^{\lambda})\|^2   \\
+ 2\lambda^2 \E[\|H(\theta_n^{\lambda}, X_{n+1}) - h(\theta_{n}^{\lambda})\|^2 |\theta_{n}^{\lambda},\bar{\theta}_{n}^{\lambda} ] + 2\lambda^2 \| h(\theta_{n}^{\lambda}) - h(\bar{\theta}_{n}^{\lambda})\|^2 \,,
\end{multline*}
where $\tilde{a}$ is defined in \eqref{eq:definition-tilde-a}.
 Hence, for $\lambda \le 1/(a+L_1)$,
 \begin{multline*}
  \E[\|\theta_{n+1}^{\lambda} - \bar{\theta}_{n+1}^{\lambda}\|^2  |\theta_{n}^{\lambda},\bar{\theta}_{n}^{\lambda} ] \le
 (1-\lambda\tilde{a})\|\theta_n^{\lambda} - \bar{\theta}_n^{\lambda}\|^2
\\+ 2\lambda^2 \E[\|H(\theta_n^{\lambda}, X_{n+1}) - \E[H(\theta_n^{\lambda}, X_{n+1})]|\theta_{n}^{\lambda},\bar{\theta}_{n}^{\lambda}] \|^2 |\theta_{n}^{\lambda},\bar{\theta}_{n}^{\lambda} ]
 \end{multline*}
Thus, due to Lemma \ref{lem:useful-conditional-expectation},
 \begin{align*}
  &\E[\|\theta_{n+1}^{\lambda} - \bar{\theta}_{n+1}^{\lambda}\|^2  |\theta_{n}^{\lambda},\bar{\theta}_{n}^{\lambda} ]\\  &\le
 (1-\lambda\tilde{a})\|\theta_n^{\lambda} - \bar{\theta}_n^{\lambda}\|^2
+ 8\lambda^2 \E[\|H(\theta_n^{\lambda}, X_{n+1}) - H(\theta_n^{\lambda}, \E[X_{n+1}|\theta_{n}^{\lambda},\bar{\theta}_{n}^{\lambda}]) \|^2 |\theta_{n}^{\lambda},\bar{\theta}_{n}^{\lambda} ] \\
 &\le
 (1-\lambda\tilde{a})\|\theta_n^{\lambda} - \bar{\theta}_n^{\lambda}\|^2
+ 8\lambda^2 L_2^2(1+\|\theta_n^{\lambda}\|)^2 \operatorname{Var}_{\mathcal{W}}(X_0) \\
%& \times \E[(1+\|X_{n+1}\|+\|\E[X_{n+1}]\|)^{2\rho}\|X_{n+1}- \E[X_{n+1}]\|^2]
\end{align*}
which implies that
\begin{equation*}
 \E[\|\theta_{n+1}^{\lambda} - \bar{\theta}_{n+1}^{\lambda}\|^2] \le  8\lambda L_2^2(1+\sup_{n\ge0}\E[\|\theta_n^{\lambda}\|^2])]\operatorname{Var}_{\mathcal{W}}(X_0)\frac{1}{\tilde{a}},
\end{equation*}
where
\[
\operatorname{Var}_{\mathcal{W}}(X_0):= \E\big[ \big(1+\| X_0\| + \|\E[X_0]\|\big)^{2\rho}\| X_0- \E[X_0]\|^2\big].
\]
Denote by $\bar{c} = \sqrt{8 L_2^2(1+c_0)]\operatorname{Var}_{\mathcal{W}}(X_0)\frac{1}{\tilde{a}}}$, one obtains $W_2(\mathrm{Law}(\theta^{\lambda}_n),\mathrm{Law}(\overline{\theta}^{\lambda}_n)) \leq \bar{c}\lambda^{1/2}$. Then, together with Theorem \ref{thm6}, the following result can be obtained
\begin{eqnarray*}
W_2(\mathrm{Law}(\theta^{\lambda}_n),\pi)&\leq&
W_2(\mathrm{Law}(\theta^{\lambda}_n),\mathrm{Law}(\overline{\theta}^{\lambda}_n))+
W_2(\mathrm{Law}(\overline{\theta}^{\lambda}_n),\pi_{\lambda})+
W_2(\pi_{\lambda},\pi) \\
&\leq& \bar{C}[\lambda^{\frac{1}{2}}+\rme^{-a\lambda n}],
\end{eqnarray*}
where $\bar{C} = \max\{\bar{c}, c_1, c\}$. For any $0<\epsilon <1/2$, by letting $\bar{C}\lambda^{\frac{1}{2}} < \epsilon /2 $, and $\bar{C}\rme^{-a\lambda n} \leq \epsilon/2$, one obtains $\lambda < c_1 \epsilon^2$ and $n > c_2\epsilon^{-2}\ln(1/\epsilon^2)$ with $c_1 = (4\bar{C})^{-1}$, $c_2 = (ac_1 )^{-1}(\ln(2\bar{C}) +1)$.
\end{proof}

\appendix{}

\section{Technical results} \label{sec_app}

\begin{lemma}\label{lem:maximal}
Let $(X_i)_{i \in \nset}$ be a sequence of random variables such that for some $p > 0$,
$M = \sup_{i \in \nset} \E[\|X_i\|^p] < \infty$.  Then for $0 < r < p$,  $\E\left[ \sup\nolimits_{1 \le i \le j}\|X_i\|^r \right] \le j^{r/p} M^{r/p}$.
\end{lemma}
\begin{proof}
One has
	\[\E^{p/r}\left[ \sup_{1 \le i \le j}\|X_i\|^r \right] \le \E \left[ \sup_{1 \le i \le j}\|X_i\|^p \right]
\leq  \E \left[ \sum_{i=1}^j \|X_i\|^p \right] \le jM,\]
by Jensen's inequality.
\end{proof}

\begin{lemma}\label{lem:useful-conditional-expectation}
Let $\mathcal{G},\mathcal{H}\subset\mathcal{F}$
	be sigma-algebras. Let $p\geq 1$. Let $X,Y$ be $\mathbb{R}$-valued random variables in $L^p$ such that $Y$ is measurable with
	respect to $\mathcal{H}\vee\mathcal{G}$.
	Then
	\[
	\CPE[1/p]{\|X-\CPE{X}{\mathcal{H}\vee\mathcal{G}}\|^p}{\mathcal{G}}
	\leq 2 \CPE[1/p]{\| X-Y\|^p}{\mathcal{G}}.
	\]
\end{lemma}
\begin{proof} See \cite[Lemma~{6.1}]{chau:kumar:2018}.
\end{proof}

\begin{lemma}\label{trivial}
	Let $x,\, y\in\mathbb{R}^{d}$, then
\[
	\sum_{\substack{i+j+k=p \\ \{i\neq p-1\}\cap\{j\neq1\}}}\frac{p!}{i!j!k!}\|x\|^{2i}\big(2\langle x, y \rangle\big)^j \|y\|^{2k} \le
\sum_{\substack{k=0 \\ k\neq 1}}^{2p}\binom{2p}{k}\|x\|^{2p-k}\|y\|^{k}
\]
\end{lemma}

\begin{proof}
Note that
\begin{align}\label{inequality}
 &\sum_{\substack{i+j+k=p \\ \{i\neq p-1\}\cap\{j\neq1\}}}\frac{p!}{i!j!k!}\|x\|^{2i}\big(2\langle x, y \rangle\big)^j \|y\|^{2k}
 &\le
  \sum_{\substack{i+j+k=p \\ \{i\neq p-1\}\cap\{j\neq1\}}}\frac{p!}{i!j!k!}\|x\|^{2i}\big(2\|x\| \|y\|\big)^j \|y\|^{2k}.
\end{align}

Moreover,
\begin{align*}
\sum_{k=0}^{2p}\binom{2p}{k}\|x\|^{2p-k}\|y\|^{k} =& (\|x\|+\|y\|)^{2p} = (\|x\|^2 + 2\|x\|\|y\| + \|y\|^2)^p\\ = &  \sum_{i+j+k=p}\frac{p!}{i!j!k!}\|x\|^{2i}\big(2\|x\| \|y\|\big)^j \|y\|^{2k}.
\end{align*}
Consequently,
\begin{align} \label{equality}
\sum_{\substack{k=0 \\ k\neq 1}}^{2p}\binom{2p}{k}\|x\|^{2p-k}\|y\|^{k} =&   \sum_{\substack{i+j+k=p \\ \{i\neq p-1\}\cap\{j\neq1\}}}\frac{p!}{i!j!k!}\|x\|^{2i}\big(2\|x\| \|y\|\big)^j \|y\|^{2k}.
\end{align}
Thus, in view of \eqref{inequality} and \eqref{equality}, the desired result is obtained.
\end{proof}

\begin{lemma}\label{miyazaki}
For each integer $r\geq 1$, $\E[\|\xi_1\|^{2r}]\leq 2^{2r}d^r r^{3r/2}$.
\end{lemma}
\begin{proof}
Let $\zeta_1,\ldots,\zeta_d$ denote the coordinates of $\xi_1$.
It is well-known that $\E[\zeta_1^{2r}]=2^r\Gamma([2r+1]/2)/\sqrt{\pi}$. Clearly,
\begin{align*}
 \|\xi_1\|_{2r} & =
%\E^{1/2r}(\zeta_1^2+\ldots+\zeta_d^2)^r
\leq \left(\sum\nolimits_{i=1}^d \E^{1/r}[\zeta_i^{2r}]  \right)^{1/2}= (2d\Gamma^{1/r}([2r+1]/2)\pi^{-1/(2r)})^{1/2} \\
&
% \leq \sqrt{2d}\Gamma^{1/2r}([2r+1]/2)
\leq\sqrt{2d}\Gamma^{1/2r}(r+1) \pi^{-1/4r} \leq \sqrt{2d}(\sqrt{2\pi}r^{r+1/2}   \rme^{-r}  \rme^{1/(12 r)})^{1/2r} \pi^{-1/4r}
\end{align*}
where an estimate for the gamma function from \cite{robbins} is used in the last inequality.
Continuing in a somewhat rough way, one obtains
\begin{equation*}
\|\xi_1\|_{2r} \leq  2\sqrt{d} r^{1/2+(1/4r)} \rme^{-1/2} \rme^{1/2} \leq  2\sqrt{d} r^{3/4}.
\end{equation*}.
\end{proof}

\section{Proof of a pivotal inequality}\label{sec_iniga}

In this section we prove the analogues of two moment inequalities from \cite{laci1}
for conditional $L$-mixing processes. One of these has already been shown in \cite{chau:kumar:2018}
but only under specific assumptions on the filtration.
Our proofs (which mostly take place in continuous time) follow closely the arguments of \cite{laci1}. There are, however,
a number of small modifications that need to be pointed out.

We consider a continuous-time filtration $(\mathcal{R}_t)_{t\in\mathbb{R}_+}$ as well as a decreasing family of sigma-fields
$(\mathcal{R}_t^+)_{t\in\mathbb{R}_+}$. We assume that $\mathcal{R}_t$ is
independent of $\mathcal{R}_t^+$, for all $t\in\mathbb{R}_+$.

We consider an $\mathbb{R}^{d}$-valued continuous-time stochastic process $(X_{t})_{t\in\mathbb{R}_+}$
which is progressively measurable (i.e.\ $X:[0,t]\times\Omega\to\mathbb{R}^d$ is $\mathcal{B}([0,t])\otimes\mathcal{R}_t$-measurable
for all $t\in\mathbb{R}_+$).
%We call the given process $L^r$-\emph{bounded} for some $r\geq 1$ if
%\[
%\sup_{t\in\mathbb{R}_+}\E[|X_{t}|^{r}]<\infty.
%\]

From now on we assume that $X_{t}\in L^{1}$, $t\in\mathbb{R}_{+}$.
We define the quantities
\begin{align*}
	\tilde{M}_r^i &:= \mathrm{ess.}\sup_{t \in\mathbb{R}_+} \CPE[1/r]{|X_{t}^i|^{r}}{\mathcal{R}_{0}},	\\ \tilde{\gamma}^i_r(\tau) &:=  \mathrm{ess.}\sup_{t\geq\tau}
	\CPE[1/r]{|X_{t}^i- \CPE{X_{t}^i}{{\mathcal{R}_{t-\tau}^+\vee \mathcal{R}_0}}|^r}{\mathcal{R}_{0}},\ \tau\in\mathbb{R}_+ ,
\end{align*}
and set $	M_r := \sum_{i=1}^d \tilde{M}_r^i$, $\tilde{\Gamma}_r^i := \sum_{\tau=0}^{\infty}\gamma^i_r(\tau)$ and $\Gamma_r := \sum_{i=1}^d \tilde{\Gamma}_r^i$  where $X_t^i$ refers to the $i$th coordinate of $X_t$.

\begin{remark} {\rm If $d=1$, $\mathcal{R}_{0}$ is trivial and $\mathcal{R}_{t}^{+}$, $t\in\mathbb{R}_{+}$ is right-continuous
then we get back to the setting of \cite{laci1}. It is shown in
Lemma 9.1 of \cite{laci1} that the (non-random) function
$\tau\to\gamma_{r}(\tau)$, $\tau\in\mathbb{R}_{+}$ is measurable hence $\overline{\Gamma}_{r}:= \int_{0}^{\infty}\gamma_r(\tau)\ \rmd \tau$
can be defined. \cite[Theorems 1.1 and 5.1]{laci1} formulate inequalities in terms of $\overline{\Gamma}_{r}$ instead of
$\Gamma_{r}$.

%Notice, however, that Lemma 2.1 of \cite{laci}
%ensures
%\[
%{\Gamma}_{r}\leq \gamma_{r}(0)+2\overline{\Gamma}_{r}\leq 2M_{r}+2\overline{\Gamma}_{r}
%\]
%so inequalities in terms of $\Gamma_{r}$ immediately
%imply similar inequalities in terms of $M_{r}+\overline{\Gamma}_{r}$. Consequently, Theorems \ref{elso} and \ref{masodik}
%below respectively retrieve slightly weaker versions of Theorems 1.1 and 5.1 of \cite{laci}.

We could attempt to define $\overline{\Gamma}_{r}$ for general $\mathcal{R}_{0}$ as a random variable but it requires
further assumptions and tedious arguments which we do not pursue here. We stay with $\Gamma_{r}$ which is easier
to handle and it suffices for our purposes.}
\end{remark}

\begin{theorem}\label{elso} Let $(X_{t})_{t \in \rset_+}$ be $L^{r}$-bounded for some $r\geq 2$
and let $M_{r}+\Gamma_{r}<\infty$ a.s.
Assume $\CPE{X_{t}}{\mathcal{R}_{0}}=0$ a.s.\ for $t\in\mathbb{R}_+$.
Let $f:[0,T]\to\mathbb{R}$ be $\mathcal{B}([0,T])$-measurable with $\int_{0}^{T}f_{t}^{2}\, \rmd t<\infty$. Then there is a constant $C(r)$ such that
\begin{equation}
\CPE[1/r]{\left| \int_{0}^{T} f_t X_t\, \rmd t \right|^r}{\mathcal{R}_{0}}
\leq C(r)\left( \int_{0}^{T} f_t^{2}\, \rmd t \right)^{1/2} [{M}_r +\Gamma_r],
\end{equation}
almost surely.
We can actually take $C(r)=\sqrt{r-1}$.
\end{theorem}

\begin{theorem}\label{masodik} Let the conditions of Theorem \ref{elso} hold for some $r>2$.
Then there is a constant $C'(r)$ such that
\begin{equation}\label{erd}
\CPE[1/r]{\sup_{s\in [0,T]}\left|\int_{0}^{s} f_t X_t\, \rmd t \right|^r}{\mathcal{R}_{0}}
\leq C'(r)\left( \int_{0}^{T} f_t^{2}\, \rmd t \right)^{1/2} [{M}_r +\Gamma_r],
\end{equation}
almost surely.
We can actually take
\[
C'(r)=\frac{\sqrt{r-1}}{2^{1/2}-2^{1/r}}.
\]
\end{theorem}

Note that the supremum in \eqref{erd} can be taken along rationals hence it defines a random variable.
We now state the corresponding results for conditionally $L$-mixing processes.

%\begin{definition}
%For some $r,s\geq 1$, we call $X_n$, $n\in\mathbb{N}$
%conditionally $L$-mixing of order $(r,s)$ (with respect to $(\mathcal{F}_{n},\mathcal{F}_{n}^{+})_{n\in\mathbb{N}}$) if
%it is $L^r$-bounded; $X_n$, $n\in\mathbb{N}$ is adapted to
%$\mathcal{F}_n$, $n\in\mathbb{N}$
%and the sequences  $M^n_r(X)$, $\Gamma^n_r(X)$, $n\in\mathbb{N}$
%are bounded in $L^s$. When the conditional $L$-mixing property of order $(r,s)$ holds for all $r,s\geq 1$ then we simply say that the
%random process is conditionally $L$-mixing.
%\end{definition}

%Inequality \ref{monsun} below plays a pivotal role in our proofs above.

\begin{theorem}\label{harmadik}
Let $(X_n)_{n \in \nset}$
be conditionally $L$-mixing of order $(r,1)$ for some $r\geq 2$.
Let $b_i$, $1\leq i\leq m$ be real numbers. Then for each $n\in\mathbb{N}$
\begin{equation*}
\CPE{\left| \sum_{i = 1}^{m} b_i X_{n+i} \right|^r}{\mathcal{F}_{n}}
\leq C(r)\left( \sum_{i=1}^{m} b_i^2 \right)^{1/2} [{M}^{n}_r(X)+ \Gamma^{n}_r(X)],
\end{equation*}
almost surely. If $r>2$ then also
\begin{equation}\label{monsun}
\CPE{\left|\max_{1\leq k\leq m} \sum_{i = 1}^{k} b_i X_{n+i} \right|^r}{\mathcal{F}_{n}}
\leq C'(r)\left( \sum_{i=1}^{m} b_i^2 \right)^{1/2} [{M}^{n}_r(X) +\Gamma^{n}_r(X)]
\end{equation}
holds.
\end{theorem}

We are proceeding to the proofs of the above results.
Since $\CPE{X_{t}}{\mathcal{R}^{+}_{t-\tau_{1}}\vee\mathcal{R}_{0}}$ is
 $\mathcal{R}^{+}_{t-\tau_{2}}\vee\mathcal{R}_{0}$-measurable for $t\geq\tau_{2}\geq\tau_{1}$, we obtain
from Lemma \ref{lem:useful-conditional-expectation} with
the choice $X=X_{t}$, $Y=\CPE{X_{t}}{\mathcal{R}^{+}_{t-\tau_{1}}\vee\mathcal{R}_{0}}$,
$\mathcal{H}=\mathcal{R}_{t-\tau_{2}}^{+}$, $\mathcal{G}=\mathcal{R}_{0}$ that
\begin{equation}\label{madag}
\gamma_{r}(\tau_{2})\leq 2\gamma_{r}(\tau_{1}).
\end{equation}

We need a measure-theoretical lemma about real-valued random variables $Y$ and $Z$.

\begin{lemma}\label{lem:sweat} Let $r>1$, $1/r+1/q=1$ and let $Y\in L^r$ be $\mathcal{R}_0\vee\mathcal{R}^+_s$-measurable for some $s\geq 0$.
Then for all $\mathcal{R}_s$-measurable $Z\in L^q$,
\[
\CPE{YZ}{\mathcal{R}_0}=\CPE{Y}{\mathcal{R}_0} \, \CPE{Z}{\mathcal{R}_0}.
\]
\end{lemma}
\begin{proof}
Let $A\in\mathcal{R}_0$ be arbitrary. We assume $Y= \indi{B} \indi{C}$ with $B\in\mathcal{R}_0$, $C\in\mathcal{R}_s^+$
and $Z= \indi{D}$ with $D\in\mathcal{R}_s$. Then we find that, by independence of $\mathcal{R}_s$ from $\mathcal{R}_s^+$ and
by $\mathcal{R}_0\subset\mathcal{R}_s$,
\begin{align*}
\E[\indi{A} YZ ] &= \P(C) \P(A\cap B\cap D)=\P(C)\E[\indi{A\cap B} \CPE{\indi{D}}{\mathcal{R}_0}]\\
&= \E[\indi{A} \indi{B} \P(C)  \CPE{\indi{D}}{\mathcal{R}_0}]=
\E[\indi{A} \CPE{Y}{\mathcal{R}_0} \CPE{Z}{\mathcal{R}_0}],
\end{align*}
which proves the statement for this $Y$ and $Z$. Now, by standard arguments, one can extend these to $Y=\indi{G}$
for all $G\in \mathcal{R}_0\vee\mathcal{R}^+_s$. We thus obtain the result for step functions $Y$, $Z$; then
for bounded measurable functions and finally we arrive at the general statement.
\end{proof}

Now we formulate, in the present setting, the analogue of \cite[Lemma~2.3]{laci1}.

\begin{lemma}\label{wales}
Let the assumptions of Theorem \ref{elso} be in force. Let $d=1$ and $1/r+1/q=1$. We have, for all $0\leq s\leq t$,
\[
\left|\CPE{X_t\eta}{\mathcal{R}_0}\right|\leq \gamma_r(t-s) \CPE[1/q]{|\eta|^q}{\mathcal{R}_0}
\]
for each  $\eta\in L^q$ which is $\mathcal{R}_s$-measurable.
\end{lemma}
\begin{proof}
Using Lemma \ref{lem:sweat},
\[
\CPE{X_t\eta}{\mathcal{R}_0}=\CPE{ \CPE{X_t}{\mathcal{R}_s^+\vee \mathcal{R}_0}}{\mathcal{R}_0 } \CPE{\eta}{\mathcal{R}_0}+
\CPE{(X_t-\CPE{X_t}{\mathcal{R}_s^+\vee \mathcal{R}_{0}})\eta}{\mathcal{R}_{0}}.
\]
Note that $\CPE{\CPE{X_t}{\mathcal{R}_s^+\vee \mathcal{R}_0}}{\mathcal{R}_0}= \CPE{X_t}{\mathcal{R}_0}=0$.
The conditional H\"older inequality implies that
\[
|\CPE{(X_t- \CPE{X_t}{\mathcal{R}_s^+\vee \mathcal{R}_{0}})\eta}{\mathcal{R}_{0}}|
\leq \gamma_r(t-s) \, \CPE[1/q]{|\eta|^q}{\mathcal{R}_{0}},
\]
showing the statement.
\end{proof}

\begin{proof}[Proof of Theorem \ref{elso}] First let $d:=1$.
For $t\in [0,T]$, define $I_{t}:=\int_{0}^{t}f_{s}X_{s}\, \rmd s$ and $g_{t}:=\CPE{|I_{t}|^{r}}{\mathcal{R}_{0}}$.
Following verbatim the arguments in the proof of \cite[Theorem~1.1]{laci1} we arrive at
\[
|I_{T}|^{r}=\int_{0}^{T}\int_{0}^{t}r(r-1)f_{t}X_{t}f_{s}X_{s}|I_{s}|^{r-2}\, \rmd s\,\rmd t.{}
\]
Hence, using Lemma \ref{wales}, $\lfloor t-s\rfloor\leq t-s$ and \eqref{madag},
\begin{eqnarray*}
g_{T} &\leq& \int_{0}^{T}\int_{0}^{t}r(r-1)|f_{t}f_{s}\CPE{X_{t}X_{s}|I_{s}|^{r-2}}{\mathcal{R}_{0}}|\, \rmd s\, \rmd t\\
&\leq& 	\int_{0}^{T}\int_{0}^{t}r(r-1)|f_{t}f_{s}|2\gamma_{r}(\lfloor t-s\rfloor)M_{r}g_{s}^{1-2/r}\, \rmd s\, \rmd t\\
&=&  \int_{0}^{T}g_{s}^{1-2/r} r(r-1)|f_{s}|\int_{s}^{T}|f_{t}|2\gamma_{r}(\lfloor t-s\rfloor)M_{r}\, \rmd t\, \rmd s
\end{eqnarray*}
almost surely, whereupon Lemma 2.5 of \cite{laci1} implies
\[
g_{T}^{1/r}\leq \left(\frac{1}{r/2}\int_{0}^{T}r(r-1)|f_{s}|\int_{s}^{T}|f_{t}|2\gamma_{r}(\lfloor t-s\rfloor)M_{r}\, \rmd t\, \rmd s\right)^{1/2}
\]
almost surely. The Cauchy inequality leads to
\begin{eqnarray*}
g_{T}^{1/r} &\leq& 2\sqrt{r-1}M_{r}^{1/2}\left(\int_{0}^{T}f^{2}_{s}\, \rmd s\right)^{1/4}
\left(\int_{0}^{T}\left(\int_{s}^{T}|f_{t}|\gamma_{r}(\lfloor t-s\rfloor)\, \rmd t\right)^{2}\, \rmd s\right)^{1/4}.
\end{eqnarray*}
Moreover, by the Minkowski inequality for the Hilbert space $L^{2}([0,T],\mathcal{B}([0,T]),\mathrm{Leb})$,
\begin{eqnarray*}
& & \left(\int_{0}^{T}\left(\int_{s}^{T}|f_{t}|\gamma_{r}(\lfloor t-s\rfloor)\, \rmd t\right)^{2}\, \rmd s\right)^{1/2}\\
&=& \left(\int_{0}^{T}\left(\sum_{k=0}^{\infty} \gamma_{r}(k)\int_{0}^{1}|f_{s+k+u}|\indiacc{ s+ k+u\leq T }\, \rmd u\right)^{2}\, \rmd s\right)^{1/2}\\
&\leq&
\sum_{k=0}^{\infty} \gamma_{r}(k)\left(\int_{0}^{T}\left(\int_{0}^{1}|f_{s+k+u}| \indiacc{s+k+u\leq T}\, \rmd u\right)^{2}\, \rmd s\right)^{1/2}\\
&\leq&
\sum_{k=0}^{\infty} \gamma_{r}(k)\left(\int_{0}^{T}\int_{0}^{1}f_{s+k+u}^{2} \indiacc{s+ k+u\leq T}\, \rmd u\, \rmd s\right)^{1/2}\\
&=&
\sum_{k=0}^{\infty} \gamma_{r}(k)\left(\int_{0}^{1}\int_{0}^{T}f_{s+k+u}^{2} \indiacc{ s+k+u\leq T}\, \rmd s\, \rmd u\right)^{1/2}\\
&\leq &
\sum_{k=0}^{\infty} \gamma_{r}(k)\left(\int_{0}^{1}\int_{\min\{k+u,T\}}^{T}f_{t}^{2}\, \rmd t\, \rmd u\right)^{1/2}\\
&\leq &
\sum_{k=0}^{\infty} \gamma_{r}(k)\left(\int_{0}^{T}f_{t}^{2}\, \rmd t\right)^{1/2}.
\end{eqnarray*}
Thus we finally arrive at
\[
g_{T}^{1/r}\leq 2\sqrt{r-1}M_{r}^{1/2}\left(\int_{0}^{T}f^{2}_{s}\, \rmd s\right)^{1/4}\left(\int_{0}^{T}f_{t}^{2}\, \rmd t\right)^{1/4}\Gamma_{r}^{1/2},
\]
which allows to conclude since $\sqrt{\Gamma_r M_r}\leq [\Gamma_r+ M_r]/2$.
Now let $d$ be arbitrary. Applying the one-dimensional result componentwise gives the result, noting the the Minkowski inequality and
the definitions of $M_r$, $\Gamma_r$ as sums of $M_r^i$, $\Gamma_r^i$, respectively.
\end{proof}

\begin{proof}[Proof of Theorem \ref{masodik}] Again, let $d:=1$.
Let $\mathcal{I}:=\{(a,b):0\leq a<b\leq T,\ \int_{a}^{b}f_{s}^{2}\, \rmd s>0\}$ and
define, for $(a,b)\in\mathcal{I}$,
\[
K_{a,b}:=\frac{\sup_{t\in [a,b]}|\int_{a}^{t}f_{s}X_{s}\, \rmd s|^{r}}{\int_{a}^{b}f_{s}^{2}\, \rmd s}	
\]
which is a random variable since the supremum can be taken along the rational numbers.
Set $M_{a,b}:=\CPE[1/r]{K_{a,b}}{\mathcal{R}_{0}}$. Define,
furthermore
\[
M^{*}_{T}:=\mathrm{ess.}\sup_{(a,b)\in\mathcal{I}}M_{a,b}.
\]
Noting Theorem \ref{elso} and following verbatim the arguments in the proof of Theorem 5.1 in \cite{laci1}
we arrive at
\[
M_{T}^{*}\leq \frac{\sqrt{r-1}[M_{r}+\Gamma_{r}]}{\sqrt{2}}+\frac{2^{1/r}}{\sqrt{2}}M_{T}^{*}
\]
almost surely, which implies
\[
M_{T}^{*}\leq \frac{\sqrt{r-1}[M_{r}+\Gamma_{r}]}{2^{1/2}-2^{1/r}},
\]
showing the statement. The case $d>1$ follows by a componentwise application of the one-dimensional result.
\end{proof}

\begin{proof}[Proof of Theorem \ref{harmadik}] Fix $n\in\mathbb{N}$.
We define the continuous-time process $\tilde{X}_{0}:=X_{n}$,
\[
\tilde{X}_{t}:=X_{n+k+1}\mbox{ for }k<t\leq k+1,\ k\in\mathbb{N}.
\]
Set $\mathcal{R}_{t}:=\mathcal{F}_{n+\lceil t\rceil}$ and $\mathcal{R}_{t}^{+}:=\mathcal{F}^{+}_{n+\lceil t\rceil}$
for $t\in\mathbb{R}_{+}$. Notice that, for $\tau\in\mathbb{N}$, $\gamma_{r}(\tau)$ calculated for
$(\tilde{X}_{t},\mathcal{R}_{t},\mathcal{R}^{+}_{t})_{t\in\mathbb{R}_{+}}$
coincides with $\gamma_{r}^{n}(\tau,X)$ as defined in \eqref{eq:definition-gamma} and \eqref{eq:definition-Gamma} for
$({X}_{n},\mathcal{F}_{n},\mathcal{F}^{+}_{n})_{n\in\mathbb{N}}$. Similarly, $M_{r}$ calculated for $\tilde{X}$
coincides with $M^{n}_{r}(X)$. Let $T:=m$, define $f_{t}:=b_{i}$, $i-1<t\leq i$, $i=1,\ldots,m$ and $f_{0}=0$.
Clearly,
\[
\int_{0}^{T}f_{t}\tilde{X}_{t}\, \rmd t=\sum_{i = 1}^{m} b_i X_{n+i}
\]
An application of Theorems \ref{elso} and \ref{masodik} to $\tilde{X}$ yield the result.
\end{proof}

%\begin{equation}
%\CPE[1/r]{\left| \int_{0}^{T} f_t X^i_t\, \rmd t \right|^r}{\mathcal{R}_{0}}
%\leq C(r)\left( \int_{0}^{T} f_t^{2}\, \rmd t \right)^{1/2} \sqrt{{M^i}_r
%\Gamma^i_r},
%\end{equation}
%And by Minkovski inequality (denoting $(e_i)_{i=1}^m$ the canonical vectors
%of $\rset^m$)
%\begin{equation}
%\CPE[1/r]{\left| \int_{0}^{T} f_t X_t\, \rmd t \right|^r}{\mathcal{R}_{0}}=
%\CPE[1/r]{\left| \int_{0}^{T} f_t X^i_t e_i\, \rmd t
%\right|^r}{\mathcal{R}_{0}}=
%\leq C(r) \left( \int_{0}^{T} f_t^{2}\, \rmd t \right)^{1/2} \sum_{i=1}^m
%\sqrt{{M^i}_r \Gamma^i_r},
%\end{equation}
%Since $\sqrt{ab} \leq (a+b)/2$
%This is
%\[
%\leq (1/2) C(r) \left( \int_{0}^{T} f_t^{2}\, \rmd t \right)^{1/2}
%\{\sum_{i=1}^m M_i^r + \sum_{i=1=^m} \Gamma^i_r \}
%\]
%It suffices to modify slightly the definitions of the L-mixing coefficients
%to obtain a neat result
%\[
%M_r^n(X)= \mathrm{ess}\sup_{\theta\in D}\sup_{m \in\nset}
%              \sum_{i=1}^m \CPE[1/r]{|U^i_{n+m}(\theta)|^r}{\mathcal{F}_n}
%\]
%And idem for $\gamma$.

\end{document}